\documentclass[11pt,a4paper]{article}
\usepackage[left=2.0cm, top=2.0cm,bottom=2.0cm,right=2.0cm]{geometry}
\usepackage{mathtools,amssymb,amsthm,mathrsfs,calc,graphicx,xcolor,cleveref,dsfont,tikz,pgfplots,bm}
\usepackage[british]{babel}
\usepackage{amsfonts}              
\usepackage[T1]{fontenc}

\usepackage{verbatim}

\numberwithin{equation}{section}

\usepackage{float}
\usepackage[font={small,it}]{caption}

\newtheorem {theorem}{Theorem}[section]
\newtheorem {proposition}[theorem]{Proposition}
\newtheorem {lemma}[theorem]{Lemma}
\newtheorem {corollary}[theorem]{Corollary}

{\theoremstyle{definition}
	
}
{\theoremstyle{theorem}
	\newtheorem {remark}[theorem]{Remark}
	
}

\newcommand{\Var}{\operatorname{Var}}
\newcommand{\Cov}{\operatorname{Cov}}

\newcommand{\dint}{\textup{d}}

\def\1{\mathbf{1}}

\def\EE{\mathbb{E}}

\def\NN{\mathbb{N}}

\def\RR{\mathbb{R}}

\def\cB{\mathcal{B}}

\def\cE{\mathcal{E}}

\def\cV{\mathcal{V}}

\def\ll{\langle \!\langle}
\def\gg{\rangle\!\rangle}

\usepackage[normalem]{ulem}

\makeatletter
\let\@fnsymbol\@alph
\makeatother

\allowdisplaybreaks

\begin{document}

	
	\title{\bfseries Second-order properties for planar Mondrian tessellations}
	
	\author{Carina Betken\footnotemark[1],\; Tom Kaufmann\footnotemark[2],\; Kathrin Meier\footnotemark[3]\;\; and Christoph Th\"ale\footnotemark[4]}
	
	
	\date{}
	\renewcommand{\thefootnote}{\fnsymbol{footnote}}
	\footnotetext[1]{Ruhr University Bochum, Germany. Email: carina.betken@rub.de}
	
	\footnotetext[2]{Ruhr University Bochum, Germany. Email: tom.kaufmann@rub.de}
	
	\footnotetext[3]{Ruhr University Bochum, Germany. Email: kathrin.meier@rub.de}
	
	\footnotetext[4]{Ruhr University Bochum, Germany. Email: christoph.thaele@rub.de}
	
	\maketitle
	\begin{abstract} \noindent In this paper planar STIT tesselations with weighted axis-parallel cutting directions are considered. They are known also as weighted planar Mondrian tesselations in the machine learning literature, where they are used in random forest learning and kernel methods. Various second-order properties of such random tessellations are derived, in particular, explicit formulas are obtained for suitably adapted versions of the pair- and cross-correlation functions of the length measure on the edge skeleton and the vertex point process. Also, explicit formulas and the asymptotic behaviour of variances are discussed in detail.\\[2mm]
	\end{abstract}
	
	\section{Introduction}
	
	Let $W\subset\RR^2$ be a convex polygon, fix some time parameter $t>0$, and let $\Lambda$ be a locally finite and translation-invariant measure on the space of lines in the plane. A random STIT tessellation in the window $W$ can be constructed by the following random process. We start by assigning to $W$ a random lifetime, which is exponentially distributed with parameter $\Lambda([W])$, where $[W]$ denotes the set of lines having non-empty intersection with $W$. When the lifetime of $W$ has expired and is less or equal to $t>0$, a random line $L$ that hits $W$ is selected according to the distribution $\Lambda(\,\cdot\, \cap [W])/\Lambda([W])$ and splits $W$ into two smaller polygons $W\cap L^+$ and $W\cap L^-$, where $L^\pm$ are the two closed half-spaces bounded by $L$. The construction is now repeated independently within $W\cap L^+$ and $W\cap L^-$ until the time threshold $t$ is reached. The STIT tessellation of $W$ with lifetime $t>0$ is understood to be the union of the splitting edges constructed as above until time $t>0$, and is denoted as $Y_t(W)$ (see Figure \ref{fig:Construction}). These edges are referred to as maximal edges of $Y_t(W)$. For a formal construction of STITs we refer the reader to  \cite{mecke2008global, NW2005}.
	It has been shown in \cite{NW2005}, that as a consequence of the choice of the exponential lifetime distributions, the construction just described can be consistently extended into the whole plane. This means that there exists a random tessellation $Y_t$ of $\RR^2$ with the property that its restriction to a polygon $W$ has the same distribution as $Y_t(W)$. We also note that the random tessellation $Y_t$ is stationary and locally finite (with the latter meaning that any compact set in $\RR^2$ will almost surely be tessellated by $Y_t$ into an at most  finite number of subsets), since the line measure $\Lambda$ has been assumed to be translation-invariant and locally finite. For a STIT tessellation $Y_t(W)$ we denote by $\mathcal{E}_t$ the random edge length measure on $Y_t(W)$ and by $\mathcal{V}_t$ the related vertex process, i.e., the point process of  intersection points of edges in $Y_t(W)$. 
	
	In the existing literature special attention has been payed to planar STIT tessellation in the so-called isotropic case, as well as their higher-dimensional analogues. The isotropic case appears if the line measure $\Lambda$ is taken to be not only translation-invariant but also invariant under rotations, that is, $\Lambda$ is a constant multiple of the unique Haar measure on the space of lines. The reason for this particular choice is that only in this case one can rely on classical integral geometry and obtain the most explicit 
	\begin{figure}[H]
		\centering
		\includegraphics[scale=0.45]{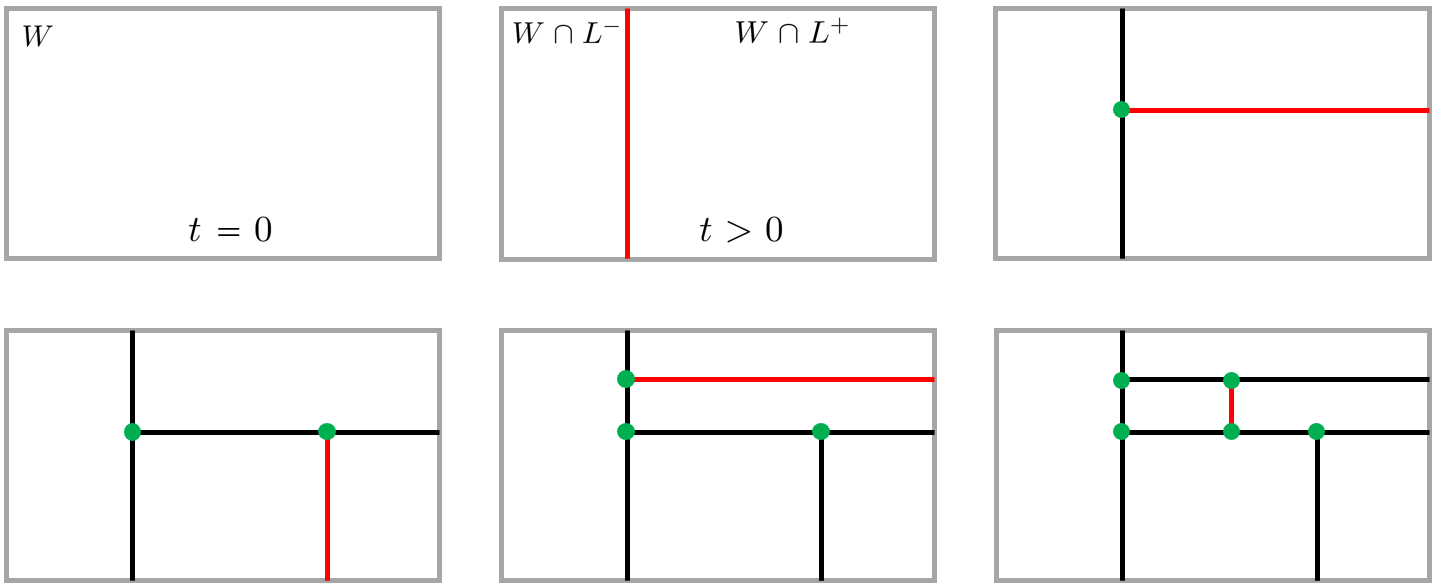}
		\caption{Construction in a rectangular window $W$  of an axis-parallel STIT or (weighted) Mondrian tessellation. The splitting edges are shown in red, the intersection points in green.}
		%
		\label{fig:Construction}
	\end{figure}
	formulas for second-order parameters related to such tessellations, see for instance \cite{RT13} for second order properties of different types of tessellations, including the isotropic case of planar STIT tessellations.  
	
	On the other hand, it has turned out that there is another particular class of STIT tessellations, called Mondrian tessellations, for which a second-order description is desirable, since such tessellations have found numerous applications in machine learning. Reminiscent of the famous paintings of the Dutch modernist painter Piet Mondrian, the eponymous tessellations are a version of STIT tessellations with only axis-parrallel cutting directions. Originally established by Roy and Teh \cite{RoyTehMondrian}, Mondrian tessellations have been shown to have multiple applications in random forest learning \cite{LRT2014, LRT2016} and kernel methods \cite{Balog2016}. Both random forest learners and random kernel approximations based on the Mondrian process have shown significant results, especially as they are substantially more adapted to online-learning (i.e., the ability to incorporate new data into an existing model without having to completely retrain it) than many other of their tessellation-based counterparts. This is due to the self-similarity of Mondrian tessellations, which stems from their defining characteristic of being iteration stable (see \cite{NW2005}), and allows to obtain explicit representations for many conditional distributions of Mondrian tessellations. This property allows a tessellation-based learner to be re-trained on new data without having to newly start the training process and is thus considerably more efficient on large data sets. These methods have recently been carried over back to their origin in stochastic geometry, i.e., to general STIT tessellations \cite{Ge2019, Oreilly2020}.
	
	Formally, by a weighted planar Mondrian tessellation we understand a planar STIT tessellation whose driving line measure $\Lambda$ is concentrated on a set of lines having axis-parallel directions, one with weight $p\in(0,1)$, the other with weight $(1-p)$ (see Figure \ref{fig:WeightedMondrian}). Such tessellations are in the focus of the present paper. While mean values for Mondrian and more general STIT tessellations can be derived by means of classical translative integral geometry, for second-order properties this is only the case for isotropic STIT tessellations, since in this case integralgeometric methods with respect to the full group of rigid motions become an essential tool in the analysis. The purpose of this paper is to close this gap partially and to study second-order properties of weighted planar Mondrian tessellations.
	
	In the following Section \ref{sec:MainResults} our main results will be presented. We will provide variance formulas depending on the parameter $p$ for both the number of maximal edges in $Y_t(W)$ and the weighted total edge length of $Y_t(W)$. We will then present what we call the Mondrian pair-correlation functions for $\mathcal{E}_t$ and $\mathcal{V}_t$ and their (Mondrian) cross-correlation function. As will be explained in more detail in the relevant section, by Mondrian pair-correlation functions we refer to versions of the classic pair-correlation function as the derivative of Ripley's K-Function, which has been adapted to the non-isotropic nature of the underlying driving measure (analogously for the cross-correlation function). The remaining sections of this paper are then dedicated to proving the main results, with Section \ref{sec:Variances} proving the variance results, Section \ref{sec:Aux} providing some auxiliary results, Section \ref{sec:EdgePair} and Section \ref{sec:VertexPair} deriving the Mondrian pair-correlation functions for $\mathcal{E}_t$ and $\mathcal{V}_t$, respectively, and Section \ref{sec:CrossCorr} showing their (Mondrian) cross-correlation function.

\begin{figure}
	\centering
	\begin{minipage}{0.32\textwidth}
		\centering
		\includegraphics[width=0.95\textwidth]{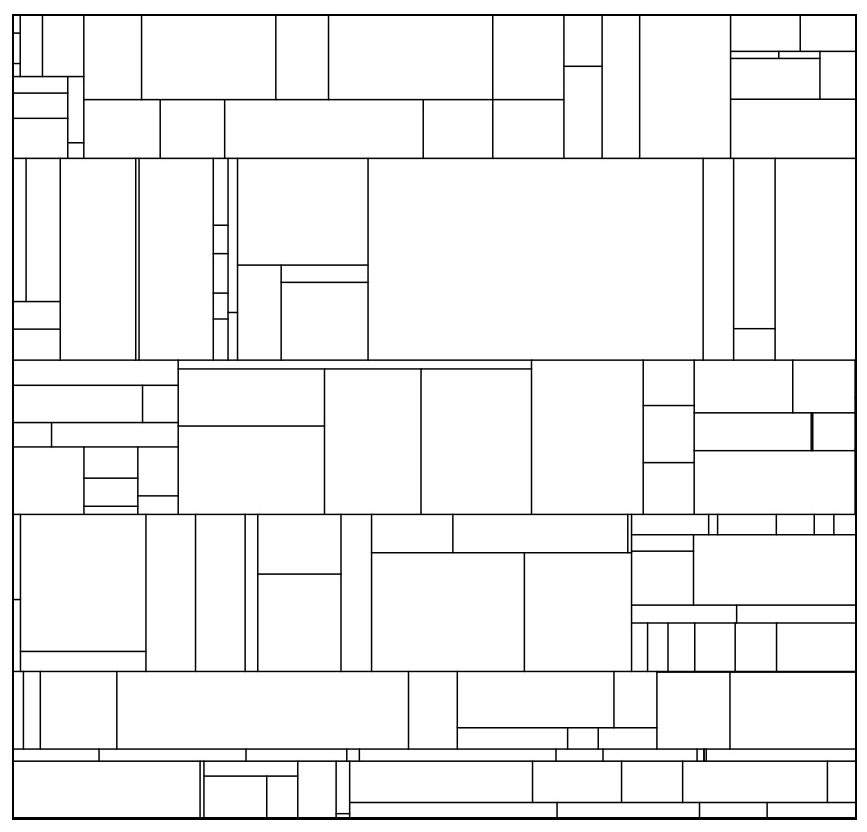} 
	\end{minipage}\hfill
	\begin{minipage}{0.32\textwidth}
		\centering
		\includegraphics[width=0.95\textwidth]{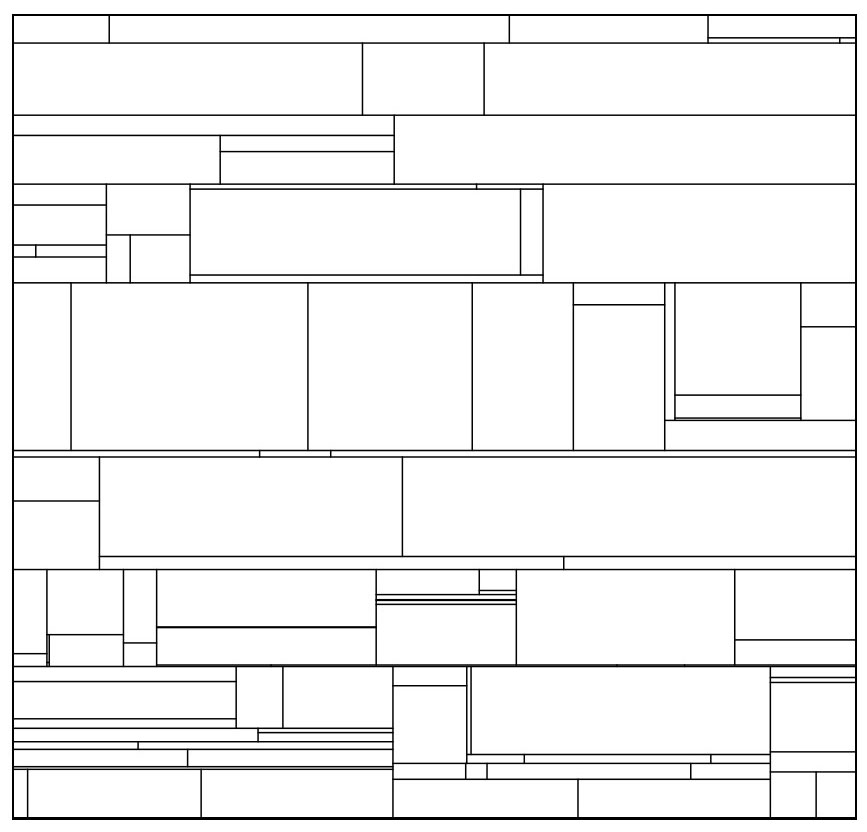} 
	\end{minipage}\hfill 
	\begin{minipage}{0.32\textwidth}
		\centering
		\includegraphics[width=0.95\textwidth]{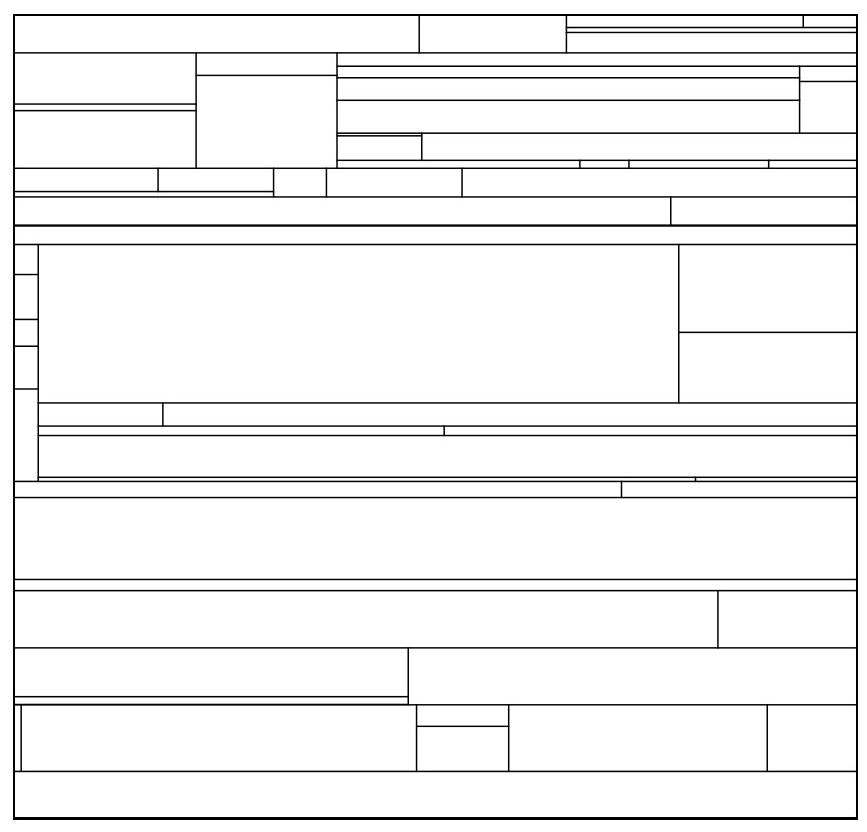} 
	\end{minipage}
	\caption{Planar Mondrian tesselations on $[0,1]^2$ with $p=0.5$, $p=0.75$, and $p=0.9$, (from left to right).}
	\label{fig:WeightedMondrian}
\end{figure}

\section{Main results}\label{sec:MainResults} 

\subsection{Notation} \label{subsec:Notation}

Let $\RR$ denote the real line and $\RR^2$ the Euclidean plane equipped with their {respective} Borel $\sigma$-fields, and write $\ell_1(\cdot)$ and $\ell_2(\cdot)$ {for the corresponding} Lebesgue measure on each space. Given a topological space $\mathbb{X}$ and a measure $\mu$ on $\mathbb{X}$, we denote by $\mu^{\otimes k}$, $k\in\NN$, its $k$-fold product measure. Further, for a set $A \subset \mathbb{X}$ we write ${\bf 1}_A( \cdot)$ for its indicator function and $\mathfrak{B}(\mathbb{X})$ the Borel $\sigma$-field on $\mathbb{X}$. For a random element $X$ in $(\mathbb{X}, \mathfrak{B}(\mathbb{X}))$ we write $\mathbb{P}_X$ for its distribution, and denote by $X \overset{d}{=} Y$ equality in distribution of two $\mathbb{X}$-valued random elements $X,Y$, i.e., \ $\mathbb{P}_X=\mathbb{P}_Y$. 

Let $[\RR^2]$ be the space of lines in $\RR^2$. Equipped with the Fell topology, $[\RR^2]$ carries a natural Borel $\sigma$-field $\mathfrak{B}([\RR^2])$, see \cite[Chapter A.3]{PenroseLast}. Further, define $[\RR^2]_0$ to be the space of all lines in $\RR^2$ passing through the origin. For a line $L\in [\RR^2]$, we write $L^+$ and $L^-$ for the positive and negative half-spaces of $L$, respectively, and $L^\perp$ for its orthogonal line passing through the origin. 
For a compact set $A\subset \RR^2$ define 
$$
[A]:= \{ L \in [\RR^2] : L \cap A\neq \emptyset\} \in \mathfrak{B}([\RR^2])
$$
to be the set of all lines in the plane that have non-empty intersection with $A$. We set  $\mathcal{S}(A)$ and $\mathcal{S}(\RR^2)$ to be the finite line segments within $A$ and $\RR^2$, respectively, and denote the line segment between two points $x,y \in \RR^2$ as $\overline{x,y}$.  Furthermore, for lines $L, L_1, L_2 \in [\RR^2]$ with $\emptyset \ne L\cap L_1 \ne L$, $\emptyset \ne L\cap L_2 \ne L$, and $L_1 \ne L_2$, we define $L(L_1, L_2) = \overline{L\cap L_1, L \cap L_2}$ to be the line segment connecting the two intersection points, otherwise we set $L(L_1, L_2) = \emptyset$. For any locally finite, translation invariant measure $\Lambda$ on $[\RR^2]$ we have a unique measure $\mathcal{R}$ on $[\RR^2]_0$, called the directional measure, that allows the decomposition 
$$
\int_{[\RR^2]} g(L) \,  \Lambda(\dint L) = \int_{[\RR^2]_0} \int_{L_0^\perp} g(L_0 + z) \, \ell_1(\dint z) \,  \mathcal{R}(\dint L_0)
$$
for any non-negative measurable function $g: [\RR^2] \to \RR$, see \cite[Theorem 4.4.1]{SchneiderWeil}. Sufficient normalization is usually applied to $\mathcal{R}$ in order to gain a probability distribution, which is then referred to as the directional distribution.

\subsection{Set-up} \label{subsec:Setup}

For a weight parameter $p\in(0,1)$ we consider the  measure $\Lambda_p$ on the space $[\mathbb{R}^2]$ of lines of the following form:
\begin{eqnarray} \label{def:Lambda}
	\Lambda_p(\,\cdot\,) := p\int_\RR{\bf 1}(E_1+\tau e_2\in\,\cdot\,)\,\dint\tau + (1-p)\int_\RR{\bf 1}(E_2+\tau e_1\in\,\cdot\,)\,\dint\tau,
\end{eqnarray}
where $\{e_1,e_2\}$ is the standard orthonormal basis in $\RR^2$, $E_i={\rm span}(e_i)$ for $i\in\{1,2\}$ and we integrate with respect to the Lebesgue measure. By a Mondrian tessellation of some convex polygon $W \subset \RR^2$ with weight $p\in(0,1)$ and time parameter $t>0$ we understand a planar STIT tessellation $Y_t(W)$ with driving measure $\Lambda_p$. Given {a polygon $W \subset\RR^2$ and a Mondrian tessellation $Y_t(W)$}, $t >0$, together with a bounded, measurable functional $\phi$ 
of tessellation edges or line segments, respectively, we define the functionals $\Sigma_\phi$ and $A_\phi$ by
\begin{equation}\label{eq:FuncAllgSigma}
	\Sigma_\phi(Y_t(W)) := \sum_{e \, \in \, Y_t(W)} \phi(e),
\end{equation}

and 

\begin{equation}\label{eq:FuncAllgA}
	A_\phi(Y_t(W))=\int_{[W]} \sum_{e \, \in \,  \operatorname{Segments}(Y_t(W) \cap L)} \phi(e) \ \Lambda(\dint L),
\end{equation}
where $Y_t(W) \cap L$ for $L\in [W]$ is a one-dimensional tessellation of $W\cap L$, whose cells are simply the associated segments (intervals) of the tessellation. In the focus of the present paper are the following special cases:
\begin{itemize}
	\item[(i)] taking $\phi\equiv 1$, $\Sigma_\phi(Y_t(W))$ reduces to the number of maximal edges of $Y_t(W)$;
	\item[(ii)] taking $\phi(e)=\Lambda_p([e])$, $\Sigma_\phi(Y_t(W))$ is the weighted total edge length of $Y_t(W)$, where all maximal edges parallel to the $e_1$ direction are weighted by the factor $p$ and the maximal edges parallel to the $e_2$ direction with a factor $1-p$. In the following we write $\Sigma_{\Lambda_p}$  and  $A_{\Lambda_p}$ instead of $\Sigma_{\Lambda_p([\cdot])}$ and  $A_{\Lambda_p([\cdot])}$ for notational brevity.
\end{itemize}
We note that property (ii) follows from the observation that, for a measurable set $A\subset\RR^2$,
\begin{equation}\label{eq:Lambda}
	\Lambda_p([A]) = p\ell_{1}(\Pi_{E_{{2}}}A) + (1-p)\ell_{1}(\Pi_{E_{{1}}}A),
\end{equation}
where for a line $E\in[\mathbb{R}^2]$ we write $\Pi_E$ for the orthogonal projection onto $E$.

\subsection{Variances}

We assume the same set-up as described in the previous section and recall the definition of $\Sigma_\phi$ from \eqref{eq:FuncAllgSigma}. The expected values $\EE \Sigma_\phi$ for $\phi\equiv 1$ and $\phi(\cdot)=\Lambda_p([\,\cdot\,])$ are  known from \cite[Section 2.3]{ST-Plane} and given as follows. 

\begin{proposition}\label{Prop:Expectation}
	Consider a Mondrian tessellation $Y_t(W)$ with weight $p\in(0,1)$. Then
	\[
	\EE[\Sigma_{\Lambda_p}(Y_t(W))]= 2tp(1-p)\ell_2(W)
	\]
	and
	\[
	\EE[\Sigma_1(Y_t(W))]= t^2 p(1-p)\,\ell_2(W)+t (p\, \ell_{1}(\Pi_{E_{2}}W) + (1-p)\,\ell_{1}(\Pi_{E_{1}}W)).
	\]
	For rectangular windows $W=[-a,a] \times [-b,b]$, $a,b >0$, these expressions simplify to
	\[
	\EE[\Sigma_{\Lambda_p}(Y_t(W))]=8t p(1-p)ab \qquad \text{and} \quad \EE[\Sigma_1(Y_t(W))]=4t^2 p(1-p)ab+2t(p b+(1-p) a).
	\]
\end{proposition}

We turn now to the main result of this section, which provides a fully explicit second-order description of the functionals $\Sigma_1$ and $\Sigma_{\Lambda_p}$. We remark that for general STIT tessellation in the plane, abstract formulas have already been developed in \cite[Theorem 1]{ST-Plane} and they have been used in the isotropic case. However, specializing them explicit{ly} to Mondrian tessellations is a non-trivial task. 

\begin{theorem}\label{Thm:Variances}
	Take $ W=[-a,a]\times [-b,b] $, $a,b >0$, and let  $ Y_t(W) $  be a Mondrian tessellation with weight $p\in(0,1)$.
	\begin{itemize}
		\item[(i)] The variance of the weighted total edge length $\Sigma_{\Lambda_p}$ is given by
		\begin{align*}
			\Var(\Sigma_{\Lambda_p}(Y_t(W)))=-8abp(1-p)\Big((1-p) g_1(2at(1-p)) + p g_1(2btp) \Big)
		\end{align*}
		with $g_1(x)=\sum_{k=1}^\infty (-1)^k \frac{x^k}{k(k+1)!}$ non-positive and monotonically decreasing on $[0, \infty).$
		\item[(ii)]The variance of the number of maximal edges $ \Sigma_{1} $ of $ Y_t(W) $ in $  W $ is given by
		\begin{align*}
			&\Var(\Sigma_{1}(Y_t(W)))\\
			&=2tbp+2ta(1-p)+12abt^2p(1-p)+ 16ab  t^2p(1-p)\,\Big((1-p) g_2(2a(1-p)t)+p g_2(2btp)\Big)
		\end{align*}
		with $ g_2(x)= \sum_{k=1}^{\infty}(-1)^k \frac{x^k }{k(k+1)(k+2)!} . $
		\item[(iii)]The covariance of the two functionals in (i) and (ii) is given by
		\begin{align*}
			\Cov(\Sigma_{\Lambda_p}(Y_t(W)),\Sigma_{1}(Y_t(W)))& = 8tab p(1-p)\Big(1 -\big[(1-p)g_3(2at(1-p))+pg_3(2btp)\big]\Big)
		\end{align*}
		with $ g_3(x)= \sum_{k=1}^{\infty}(-1)^k \frac{x^k }{k(k+1)(k+1)!}. $
	\end{itemize}
\end{theorem}

Next, we specialize Theorem \ref{Thm:Variances} to the case where $ W=Q_r:=[-r,r]^2 $ is a square with side length $2r>0$ and consider the behaviour of the (co-)variances, as $r\to\infty$. In what follows we write $f(r)\asymp g(r)$ for two functions $f,g:[0,\infty)\to\mathbb{R}$, whenever $f(r)/g(r)\to 1$, as $r\to\infty$.

\begin{corollary}\label{cor:VarianceAsymptotic}
	Let $ Q_r=[-r,r]^2 $ be a square and consider a Mondrian tessellation $ Y_t(Q_r) $  with weight $p\in(0,1)$. Then, as $r\to\infty$,
	\begin{align*}
		\Var(\Sigma_{\Lambda_p }(Y_t(Q_r)))\asymp 4p(1-p)r^2\log(r),\qquad\qquad
		\Var(\Sigma_{1}(Y_t(Q_r))) \asymp 4t^2p(1-p) r^2\log(r)
	\end{align*}
	and
	\begin{align*}
		\Cov(\Sigma_{\Lambda_p}(Y_t(Q_r)), \Sigma_{1}(Y_t(Q_r)))\asymp 4tp(1-p)r^2\log(r).
	\end{align*}
\end{corollary}

\begin{remark}\label{rem:CompVariances}\rm 
	It is instructive to compare this result to the corresponding asymptotic formulas for isotropic STIT tessellations in the plane and the rectangular Poisson line process.  We therefore denote by $\Lambda_{\rm iso }$ the isometry invariant measure on the space of lines in the plane normalized in such a way that $\Lambda_{\rm iso}([[0,1]^2])={4\over\pi}$ (this is the same normalization as the one used in \cite{SchneiderWeil}).
	\begin{itemize}
		\item[(i)] For isotropic STIT tessellations it is known from \cite[Section 3.4]{ST-Plane} in combination with \cite[Theorem 3]{ST2012} that, for example,
		\begin{align*}
			\Var(\Sigma_{\Lambda_{\rm iso }}(Y_t(B_r^2))) \asymp {16\over\pi}r^2\log(r)\qquad\text{and}\qquad \Var(\Sigma_{1}(Y_t(B_r^2))) \asymp {16\over\pi}t^2r^2\log(r)
		\end{align*}
		for a disc $B_r^2$ of radius $r>0$, as $r\to\infty$.
		In addition, it can be concluded from \cite[Theorem 1]{ST-Plane} that
		\begin{align*}
			\Cov(\Sigma_{\Lambda_{\rm iso}}(Y_t(B_r^2)), \Sigma_{1}(Y_t(B_r^2)))\asymp {16\over\pi}tr^2\log(r).
		\end{align*}
		
		\item[(ii)] For the rectangular Poisson line process we consider a stationary Poisson line process $\eta_p$ in the plane with intensity measure $\lambda$ with intensity $t>0$ and directional distribution $pe_2+(1-p)e_1$. For the total edge length  $\Sigma_{t\lambda}(Q_r)$ and the number of edges $\Sigma_1(Q_r)$  
		in a square $Q_r=[-r,r]^2$  with side length $ 2r > 0 $ one has
		$$
		\Var (\Sigma_{t \lambda }(Q_r)) = 8tr^3
		$$
		and
		$$
		\Var (\Sigma_1(Q_r)) = t^3(2p(1-p)^2r^3+2p^2(1-p)r^3)+2t^2p(1-p)r^2 \asymp 2t^3r^3p(1-p).$$ 
	\end{itemize}
\end{remark}

\subsection{Correlation Function}\label{sec:CorrFct}
In \cite{ST-Plane} an explicit description of the pair-correlation function of the vertex point process of an isotropic planar STIT tessellation has been derived, while such a description for the random edge length measure can be found in \cite{ST2012}. Also the so-called cross-correlation function between the vertex process and the random length measure was computed in \cite{ST-Plane}. In the present paper we develop similar results for planar Mondrian tessellations. To define the necessary concepts, we suitably adapt the notions used in the isotropic case. We let $Y_t$ be a weighted Mondrian tessellation of $\RR^2$ with weight $p\in(0,1)$ and time parameter $t>0$, define $R_p:=[0,1-p]\times[0,p]$ and let $R_{r,p}:=rR_p$ be the rescaled rectangle with side lengths $r(1-p)$ and $rp$. In the spirit of Ripley's K-function widely used in spatial statistics \cite{SKM95}, we let $t^2K_{\cal E}(r)$ be the total edge length of $Y_t$ in $R_{r,p}$ when $Y_t$ is regarded under the Palm distribution with respect to the random edge length measure $\mathcal{E}_t$ concentrated on the edge skeleton. On an intuitive level the latter means that we condition on the origin being a typical point of the edge skeleton, see \cite[Section 4.5]{SKM95}. The classic version of Ripley's K-function considers a ball of radius $r>0$, but since our driving measure is non-isotropic, we account for that by considering $R_{r,p}$ instead. Similarly, we let $\lambda K_{\cal V}(r)$ be the total number of vertices of $Y_t$ in $R_{r,p}$, where $\lambda=t^2p(1-p)$ stands for the vertex intensity of $Y_t$ and where we again condition on the origin being a typical  vertex of the tessellation (in the sense of the Palm distribution with respect to the random vertex point process). While these functions still have a complicated form (which we will determine in the course of our arguments below), we consider their normalized derivatives -- provided these derivatives are well defined as it is the case for us. In the isotropic case, these are known as the pair-correlation functions of the random edge length measure or the vertex point process, respectively. In our case, the following normalization turns out to be most suitable:
$$
g_\mathcal{E}(r)= \frac{1}{2p(1-p)r} \frac{\dint K_\mathcal{E}(r)}{\dint r} \qquad\text{and}\qquad g_\mathcal{V}(r)= \frac{1}{2p(1-p)r} \frac{\dint K_\mathcal{V}(r)}{\dint r},
$$
where we suppress the dependence of $g_\mathcal{E}$ and $g_\mathcal{V}$ on $t>0$ for notational brevity. In contrast to the classical pair-correlation function we normalize by the factor $2p(1-p)r$ instead of $2\pi r$, as we use the adaptation of Ripley's K-function based on $R_{r,p}$ instead of a ball of radius $r$. The normalized derivatives above are what we refer to as the \textit{Mondrian pair-correlation function} of $\mathcal{E}_t$ and $\mathcal{V}_t$, respectively.
Similarly to the isotropic case, one can also define what is called a cross K-function $K_{\mathcal{E},\mathcal{V}}(r)$ and the corresponding correlation function
$$
g_{\mathcal{E},\mathcal{V}}(r)= \frac{1}{2p(1-p)r} \frac{\dint K_{\mathcal{E},\mathcal{V}}(r)}{\dint r},
$$
again suppressing the dependence of $g_{\mathcal{E},\mathcal{V}}$ on $t>0$ and refer to Equation \eqref{eq:Cross-KFunction} below for a formal description. Since it is also based on $R_{r,p}$ instead of a ball of radius $r>0$, we call  it the \textit{Mondrian cross-correlation function} of $\mathcal{E}_t$ and $\mathcal{V}_t$.
We are now prepared for the presentation of the main results of this section. We start with the Mondrian edge pair-correlation function $g_\mathcal{E}(r)$.
\begin{theorem} \label{thm:EdgePair}
	Let  $ Y_t$  be a Mondrian tessellation with weight $p\in(0,1)$ and time parameter $t>0$. Then 
	$$
	g_\mathcal{E}(r) = 1 +\frac{1}{2 t^2 r^2}\Big(\frac{1}{p^2}  +\frac{1}{ (1-p)^2 } -\frac{e^{-trp^2}}{ p^2 } -\frac{e^{-tr(1-p)^2}}{(1-p)^2 }\Big).
	$$
\end{theorem}
This result can be considered as the Mondrian counterpart to \cite[Theorem 7.1]{ST2012}, while the next theorem for the Mondrian cross-correlation function $g_{\mathcal{E},\mathcal{V}}(r)$ is the analogue of \cite[Corollary 4]{ST-Plane}.
\begin{theorem} \label{thm:CrossCorr}
	Let  $ Y_t $  be a Mondrian tessellation with weight $p\in(0,1)$ and time parameter $t>0$. Then 
	\begin{eqnarray*}
		g_{\mathcal{E},\mathcal{V}}(r)%
		&=& 1+ \frac{1}{t^2r^2p(1-p)} \Bigg[ \frac{1}{p} + \frac{1}{1-p} - \frac{1}{2tr(1-p)^3} - \frac{1}{2trp^3} \vphantom{\int_0}\\
		&&  - e^{-tr(1-p)^2} \Bigg(\frac{1}{2(1-p)} - \frac{1}{2tr(1-p)^3} \Bigg) - e^{-trp^2} \Bigg(\frac{1}{2p} - \frac{1}{2trp^3} \Bigg) \Bigg].
	\end{eqnarray*}
\end{theorem}
Finally, we deal with the Mondrian vertex pair-correlation function $g_\mathcal{V}(r)$, which in the isotropic case has been determined in  \cite[Corollary 3]{ST-Plane}. Plots for the functions $g_\mathcal{E}(r), 	g_{\mathcal{E},\mathcal{V}}(r),$ and $g_\mathcal{V}(r)$ for $t=1$ and different weights $p \in \{0.5, 0.75, 0.9\}$ are shown in Figure \ref{fig:PlotFct}. 
%

\begin{figure}[t]
	\centering
	\includegraphics[scale=0.5075]{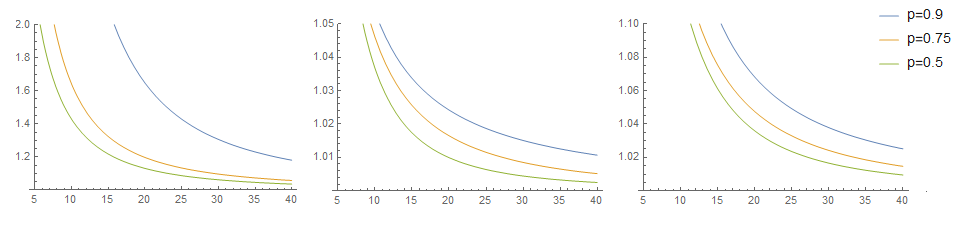}
	\caption{From left to right: $g_\mathcal{E}(r), 	g_{\mathcal{E},\mathcal{V}}(r),$ and $g_\mathcal{V}(r)$ for $t=1$.}
	\label{fig:PlotFct}
\end{figure}
\begin{theorem} \label{thm:VertexPair}
	Let  $ Y_t $  be a Mondrian tessellation with weight $p\in(0,1)$ and time parameter $t>0$. Then 
	\begin{eqnarray*}
		g_\mathcal{V}(r)
		&=& 1 +\frac{1}{t^2 r^2 p^2 (1-p)^2} \Bigg[ 4 -\frac{2}{t r p^2 } + \frac{1}{t^2 r^2 p^4 }  -\frac{2}{ t r (1-p)^2} +\frac{1}{t^2 r^2 (1-p)^4} \\
		\\
		&& - e^{-t r p^2 }\Bigg(\frac{1}{2}-\frac{1}{ t r p^2 }+\frac{1}{t^2 r^2 p^4 }\Bigg) - e^{-t r (1-p)^2 } \Bigg(\frac{1}{2} -\frac{1}{t r (1-p)^2} +\frac{1}{t^2 r^2 (1-p)^4 }\Bigg)\Bigg].
	\end{eqnarray*}
\end{theorem}

\begin{remark}\rm 
	Again, we compare our result with the correlation functions for isotropic STIT tessellations in the plane and the rectangular Poisson line process. 
	\begin{itemize}
		\item[(i)] In the isotropic case Schreiber and Thäle showed in \cite{ST2013} that the pair-correlation function of the random edge length measure $\mathcal{E}_{t}$ has the form 
		\begin{equation*}
			g_\mathcal{E}(r)  = 1 + \frac{1}{2t^2 r^2}\Big(1-e^{-\frac{2}{\pi} tr}\Big).
		\end{equation*}
		In \cite{ST2012} the same authors showed that the pair-correlation function of the vertex point process $\mathcal{V}_t$ and the cross-correlation function of the random edge length measure and the vertex point process are given by 
		\begin{equation*}
			g_{\mathcal{E},\mathcal{V}}(r) = 1+ \frac{1}{t^2 r^2}-\frac{\pi}{4t^3r^3}- \frac{e^{-\frac{2}{\pi}tr}}{2t^2r^2}\Big(1-\frac{\pi}{2tr}\Big)
		\end{equation*}
		and 
		\begin{equation*}
			g_\mathcal{V}(r) = 1 + \frac{2}{t^2r^2} - \frac{\pi}{t^3r^3} + \frac{\pi^2}{4t^4r^4} - \frac{e^{-\frac{2}{\pi}tr}}{2t^2r^2} \Big( 1 - \frac{\pi}{tr} + \frac{\pi^2}{2t^2r^2} \Big).
		\end{equation*}
		
		\item[(ii)] For the rectangular Poisson line process as given in Remark \ref{rem:CompVariances} one can use the theorem of Slivnyak-Mecke (see for example \cite[Example 4.3]{SKM95}) to deduce that the corresponding analogues of the cross- and pair-correlation functions are given by
		$$
		g_\mathcal{E}(r) = 1 + \frac{1}{tr},\qquad g_{\mathcal{E},\mathcal{V}}(r) = 1 + \frac{1}{4trp(1-p)}\qquad\text{and}\qquad g_\mathcal{V}(r)=1 + \frac{1}{2tr p^2(1-p)^2 }.
		$$
	\end{itemize}
\end{remark}

\section{Variance calculations for Mondrian tessellations}\label{sec:Variances}

\subsection{The point intersection measure and expectations}

First and second order properties of $Y_t(W)$ are typically accessible to us via the functionals $\Sigma_\phi$ and $A_\phi$ as given in \eqref{eq:FuncAllgSigma} and \eqref{eq:FuncAllgA} for specific choices of $\phi$ by means of a martingale results from \cite[Proposition 1]{ST-Plane}. In what follows we concentrate on the two specific cases $\phi\equiv 1$ and $\phi(\,\cdot\,)=\Lambda_p([\,\cdot\,])$  and define the two measures that we need to formulate expectations and variances. The first is the so-called \textit{point intersection measure} on  $\RR^2$ given by
\begin{equation*} 
	\ll \Lambda_p \cap\Lambda_p\gg(\,\cdot\,) := \int_{[\RR^2]}\int_{[L]}{\bf 1}(L\cap L'\in\,\cdot\,)\,\Lambda_p(\dint L')\Lambda_p(\dint L),
\end{equation*}
which is the main tool in the proof of Proposition~\ref{Prop:Expectation}.

\begin{proof}[Proof of Proposition~\ref{Prop:Expectation}]
	From the results in \cite[Section 2.3]{ST-Plane} we conclude that
	\begin{eqnarray} \label{eq:ExpVertexAndLampbdaAllg}
		\EE[\Sigma_{\Lambda_p}(Y_t(W))]=t \ll \Lambda_p \cap \Lambda_p \gg (W) \textup{ \quad and \quad}
		\EE[\Sigma_1(Y_t(W))]=t \Lambda_p([W])+ \frac{t^2}{2} \ll \Lambda_p \cap \Lambda_p \gg (W).
	\end{eqnarray}
	Since for measurable $A \subset \RR^2$ one has that
	\begin{align}\label{eq:IntersectionMeasure}
		\ll \Lambda_p\cap\Lambda_p\gg(A) &= \int_{[\RR^2]}\int_{[L]}{\bf 1}(L\cap L'\cap A\neq\emptyset)\,\Lambda_p(\dint L')\Lambda_p(\dint L)\notag\\
		&=2p(1-p)\int_\RR\int_\RR{\bf 1}((E_2 + \tau e_1)\cap(E_1 + \sigma e_2)\cap A\neq\emptyset)\,\dint\sigma \, \, \dint\tau\notag\\
		&=2p(1-p)\int_\RR\ell_{1}((E_2+\tau e_1)\cap A)\,\dint\tau \notag\\
		&=2p(1-p)\ell_2(A),
	\end{align}
	it follows from \eqref{eq:ExpVertexAndLampbdaAllg} that
	\[
	\EE[\Sigma_{\Lambda_p}(Y_t(W))]= 2tp(1-p)\ell_2(W)
	\]
	and  by also using \eqref{eq:Lambda}
	\[
 	\EE[\Sigma_1(Y_t(W))]=t \big(p\, \ell_{1}(\Pi_{E_2}W) + (1-p)\,\ell_{1}(\Pi_{E_1}W)\big)+ t^2 p(1-p)\,\ell_2(W).
	\]
	The expression for $W=[-a,a] \times [-b,b]$ with $a,b >0$ is now straight forward.
\end{proof}

\subsection{The segment intersection measure and variances}    

In a next step, we introduce the \textit{segment intersection measure} on the space of line segments $\mathcal{S}(\RR^2)$ by
\begin{eqnarray} \label{eq:DefSegIntersecMeasure}
	\ll (\Lambda_p\times\Lambda_p)\cap \Lambda_p\gg ( \, \cdot \,) :=\int_{[\RR^2]}\int_{[L]}\int_{[L]} \delta_{L(L_1, L_2)}(\, \cdot \,) \, \Lambda_p(\dint L_1) \, \Lambda_p(\dint L_2) \, \Lambda_p(\dint L),
\end{eqnarray}
where $\delta_x$ denotes the Dirac measure at a point $x\in\RR^2$. For a measurable $A \subset \RR^2$ we obtain that $\ll (\Lambda_p\times\Lambda_p)\cap \Lambda_p\gg (\mathcal{S}(A))$ equals
\begin{align*}
	&\int_{[\RR^2]}\int_{[L]}\int_{[L]} \delta_{L(L_1, L_2)}(\mathcal{S}(A)) \, \Lambda_p(\dint L_1) \, \Lambda_p(\dint L_2) \, \Lambda_p(\dint L)\\
	\\
	&=\int_{[\RR^2]}\int_{[L]}\int_{[L]} {\bf 1}\{L(L_1, L_2) \subset A\} \, \Lambda_p(\dint L_1) \, \Lambda_p(\dint L_2) \, \Lambda_p(\dint L)\\
	\\
	&= p\, \int_{\RR}\int_{[E_1 + \sigma e_2]}\int_{[E_1 + \sigma e_2]} {\bf 1}\{\overline{(E_1 + \sigma e_2)\cap L_1, (E_1 + \sigma e_2) \cap L_2} \subset A \} \, \Lambda_p(\dint L_1) \, \Lambda_p(\dint L_2)
	\dint \sigma\\
	\\
	& \, \, \, \, + (1-p)\, \int_{\RR}\int_{[E_2 + \sigma e_1]}\int_{[E_2 + \sigma e_1]} {\bf 1}\{\overline{(E_2 + \sigma e_1)\cap L_1, (E_2 + \sigma e_1) \cap L_2 } \subset A\}\, \Lambda_p(\dint L_1) \, \Lambda_p(\dint L_2)
	\, \dint \sigma.
\end{align*}
In order for $(E_1 + \sigma e_2) \cap L_i$ and $(E_2 + \sigma e_1) \cap L_i, \, i \in \{1,2\},$ to be non-empty, the lines $L$ and $L_i$ have to be perpendicular, i.e., for $i\in\{1,2\}$ we either have
\begin{enumerate}
	\item[i)] {$(E_1 + \sigma e_2) \cap L_i = (E_1 + \sigma e_2) \cap (E_2 + \tau_i e_1) = (\tau_i, \sigma)$ for some $\tau_i \in \RR$} or
	\item[ii)] $(E_2 + \sigma e_1) \cap L_i = (E_2 + \sigma e_1) \cap (E_1 + \tau_i e_2) = (\sigma, \tau_i)$ for some $ \tau_i \in \RR$.
\end{enumerate}
This leads to
\begin{align}\label{eq:SegmentIntersectionMeasure}
	\ll (\Lambda_p\times\Lambda_p)\cap \Lambda_p\gg (\mathcal{S}(A)) &= p(1-p)^2\, \int_{\RR}\int_{\RR}\int_{\RR} {\bf 1}\{ \overline{(\tau_1, \sigma), (\tau_2, \sigma)}\subset A \}\,\dint \tau_1\, \dint \tau_2\, \dint \sigma\notag\\
	\notag\\
	& \, \, \, \, + (1-p)p^2\, \int_{\RR}\int_{\RR}\int_{\RR} {\bf 1}\{\overline{(\sigma, \tau_1), (\sigma, \tau_2)}\subset A\}\,\dint \tau_1 \, \dint \tau_2 \,\dint \sigma.
\end{align}
Having the point and segment intersection measures at hand, we now proceed to prove Theorem \ref{Thm:Variances}.

\begin{proof}[Proof of Theorem \ref{Thm:Variances}]
	Applying \cite[Theorem 1]{ST-Plane}, we combine Equations (13) and (14) in \cite{ST-Plane} to obtain that 
	$\Var(\Sigma_{\Lambda_p}(Y_t(W)))$ is the same as 
	\begin{eqnarray*} 
		&&\int_0^t \int_{[W]} \int_{[L\cap W]} \int_{[L\cap W]} \exp(-u\Lambda_p([L(L_1,L_2)]))\ \Lambda_p(\dint L_2)\, \Lambda_p(\dint L_1)\, \Lambda_p(\dint L)\,\dint u\\
		\nonumber &=&\!\!\!\!-\int_{[W]} \int_{[L\cap W]} \int_{[L\cap W]} \frac{1}{\Lambda_p([L(L_1,L_2)])} \Big(\exp(-t\Lambda_p([L(L_1,L_2)]))-1\Big)\Lambda_p(\dint L_2)\, \Lambda_p(\dint L_1)\, \Lambda_p(\dint L)
	\end{eqnarray*}
	by Fubini's theorem.
	As outlined above, in order for $ \Lambda_{p}([L(L_1,L_2)]) $ to be non-zero we need
	$$L=E_i+\sigma e_j, \quad L_1=E_j+\tau e_i, \quad \text{ and }\quad L_2=E_j+\vartheta e_i$$
		for some $\sigma \in \Pi_{E_{j}}(W)$  and $\tau, \vartheta \in \Pi_{E_{i}}(W)$, $ i,j \in \{1,2\}$ with $ i \neq j $.
	In the integral above we then have
	\begin{equation*} 
		\Lambda_p([L(L_1,L_2)])=\ (1-p) \,|\tau-\vartheta| \ \1\{(\tau,\sigma)\in W\} \1\{(\vartheta,\sigma)\in W\}
	\end{equation*}
	for $  L=E_1+\sigma e_2 $,  $ L_1=E_2+\tau e_1 \text{ and } L_2=E_2+\vartheta e_1 $ 
	and
	\begin{equation*} 
		\Lambda_p([L(L_1,L_2)])=\ p \, |\tau-\vartheta| \ \1\{(\sigma,\tau)\in W\} \1\{(\sigma,\vartheta)\in W\},
	\end{equation*}
	if $  L=E_2+\sigma e_1 $, $ L_1=E_1+\tau e_2 \text{ and } L_2=E_1+\vartheta e_2 .$ Hence,
	\begin{align*} 
		\Var(\Sigma_{\Lambda_p}(Y_t(W)))  &= -\int_{[W]} \int_{[L\cap W]} \int_{[L\cap W]} \frac{1}{\Lambda_p([L(L_1,L_2)])} \Big(\exp(-t\Lambda_p([L(L_1,L_2)]))-1\Big)\\
		&\qquad \qquad  \qquad  \qquad \qquad \qquad \qquad  \qquad \times \Lambda_p(\dint L_2)\, \Lambda_p(\dint L_1)\, \Lambda_p(\dint L) \\
		&= \displaystyle- \Bigg(p(1-p)^2\int_{-b}^b \int_{-a}^a \int_{-a}^a \frac{1}{|\tau-\vartheta| } \Big(\exp(-t(1-p)|\tau-\vartheta|)-1\Big)\ \dint \vartheta\,\dint \tau\, \dint \sigma\\
		& \displaystyle\qquad \qquad  +(1-p)p^2\int_{-a}^a \int_{-b}^b \int_{-b}^b \frac{1}{|\tau-\vartheta| } \Big(\exp(-tp|\tau-\vartheta|)-1\Big)\dint \vartheta\,\dint \tau\, \dint \sigma\Bigg)\\
		&= \displaystyle -2\Bigg(p(1-p)^2\int_{-b}^b \int_{-a}^a \int_{-a}^\tau \frac{1}{\tau-\vartheta } \Big(\exp(-t(1-p)(\tau-\vartheta))-1\Big) \ \dint \vartheta\,\dint \tau\, \dint \sigma\\
		& \displaystyle\qquad \qquad +(1-p)p^2\int_{-a}^a \int_{-b}^b \int_{-b}^\tau \frac{1}{\tau-\vartheta } \Big(\exp(-tp(\tau-\vartheta))-1\Big) \ \dint \vartheta\,\dint \tau\, \dint \sigma\Bigg)
	\end{align*}
	by Fubini's theorem. For the inner integrals we obtain
	\begin{align*}
		&\int_{-a}^\tau \frac{1}{\tau-\vartheta } \Big(\exp(-t(1-p)(\tau-\vartheta))-1\Big) \ \dint \vartheta = \int_{0}^{\tau+a} \frac{1}{\vartheta }(\exp(-t(1-p)\vartheta)-1) \,  \dint \vartheta \\
		&= \int_{0}^{\tau+a} \frac{1}{\vartheta }\sum_{k=1}^\infty (-1)^k \frac{(t(1-p)\vartheta)^k}{k!}\  \dint \vartheta \ = \ \sum_{k=1}^\infty (-1)^k \frac{(t(1-p))^k}{k!}\frac{(\tau+a)^k}{k}
	\end{align*}
	and 
	\begin{align*}
		\int_{-b}^\tau \frac{1}{\tau-\vartheta } \Big(\exp(-tp(\tau-\vartheta))-1\Big) \ \dint \vartheta = \sum_{k=1}^\infty (-1)^k \frac{(tp)^k}{k!}\frac{(\tau+b)^k}{k},
	\end{align*}
	respectively. Furthermore,
	\begin{eqnarray*}
		\int_{-b}^b\int_{-a}^a \sum_{k=1}^\infty (-1)^k \frac{(t(1-p))^k}{k!}\frac{(\tau+a)^k}{k} \ \dint \tau \ \dint \sigma &=& 
		4ab\ \sum_{k=1}^\infty (-1)^k \frac{(2at(1-p))^k}{(k+1)!k} 
	\end{eqnarray*}
	and 
	\begin{align*}
		\int_{-a}^a \int_{-b}^b \sum_{k=1}^\infty (-1)^k \frac{(tp)^k}{k!}\frac{(\tau+b)^k}{k} \ \dint \tau \ \dint \sigma= 4ab\ \sum_{k=1}^\infty (-1)^k \frac{(2btp)^k}{(k+1)!k}.
	\end{align*}
	Recalling the definition of $g_1(x)$ from the statement of Theorem \ref{Thm:Variances} (i), proves the result. %
	We now turn to the variance of $ \Sigma_1(Y_t(W))$ in Theorem \ref{Thm:Variances} (ii).
	Combining Proposition 1 and Equation (12) in \cite{ST-Plane} gives
	\begin{align*}
	    \Var(\Sigma_1(Y_t(W)))= t \Lambda_p([W])+ 3 \int_0^t \EE[\Sigma_{\Lambda_p}(Y_s(W))] \, \dint s+ 2 \int_0^t \int_0^s  \Var(\Sigma_{\Lambda_p}(Y_u(W))) \,\dint u.
	\end{align*}
	Furthermore, Equation (8) in \cite{ST-Plane} shows that
	\begin{align}\label{eq:IntExpSigmaLambda}
	   	\int_0^t \EE[\Sigma_{\Lambda_p}(Y_s(W))] \, \dint s= \frac{t^2}{2}\ll\Lambda_p \cap \Lambda_p\gg(W). 
	\end{align}
	
	Equations~\eqref{eq:Lambda} and \eqref{eq:IntersectionMeasure} now yield
	$$	t \Lambda_p([W])= 2tbp+2ta(1-p) \textup{ \qquad and \qquad}	\frac{3}{2}t^2\ll\Lambda_p \cap \Lambda_p \gg(W)= 12t^2 ab\, p(1-p).
	$$
	For the last term we use Theorem \ref{Thm:Variances} (i) to see that
	\begin{align*} 
&	2\int_{0}^t\int_0^s \Var(\Sigma_{\Lambda_p}(Y_u(W)))\,\dint u\, \dint s\notag\\
		\quad &=-16abp(1-p)\int_{0}^t\int_0^s\Big((1-p)g_1(2a(1-p)u)+ p g_1(2bpu)\Big) \,\dint u\, \dint s.
	\end{align*}
 Now,
	\begin{align*}
		\int_{0}^t\int_0^s g_1(2a (1-p)u) \,\dint u\, \dint s=\sum_{k=1}^{\infty}(-1)^k \frac{(2a(1-p))^k}{k(k+1)!} \int_{0}^t\int_0^s u^k \,\dint u\, \dint s = t^2\, \sum_{k=1}^{\infty}(-1)^k \frac{(2a(1-p)t)^k }{k(k+1)(k+2)!}, 
	\end{align*}
	and similarly
	\begin{align*}
		\int_{0}^t\int_0^s g_1(2bup) \,\dint u\, \dint s= t^2\, \sum_{k=1}^{\infty}(-1)^k \frac{(2bpt)^k }{k(k+1)(k+2)!}. 
	\end{align*}
	 It follows that
	\begin{align*}
		&\Var(\Sigma_{1}(Y_t(W)))\\
		&= 2tbp+2ta(1-p)+12t^2abp(1-p)+ 16ab  t^2p(1-p)\,\big((1-p) g_2(2a(1-p)t)+p g_2(2btp)\big)
	\end{align*}
	proving claim the second claim. %
	To deal with the covariance in Theorem \ref{Thm:Variances} (iii), we  use \cite[Equation (1)]{ST-Plane} in combination with \eqref{eq:IntExpSigmaLambda} to see that
	\begin{align*}
		\Cov(\Sigma_{\Lambda_p}(Y_t(W)), \Sigma_{1}(Y_t(W))) 
		&=t\ll\Lambda_p  \cap \Lambda_p \gg(W)+\int_0^t \Var(\Sigma_{\Lambda_p}(Y_s(W))) \, \dint s.
	\end{align*}
 Proceeding as above, we get
	\begin{align*}
		\int_0^t \Var(\Sigma_{\Lambda_p}(Y_s(W))) \, \dint s &= -8abp(1-p)\int_0^t (1-p) g_1(2a(1-p)s)+p g_1(2bps) \, \dint s\\
		&=-8ab p(1-p) \sum_{k=1}^\infty (-1)^k \frac{(1-p)(2a(1-p))^k+p(2bp)^k}{k(k+1)!} \int_0^t s^k \, \dint s\\
		&=-8tabp(1-p)  \sum_{k=1}^\infty (-1)^k \frac{(1-p)(2at(1-p))^k+p(2btp)^k} 
		{k(k+1)(k+1)!}\\
		&=-8tabp(1-p)\Big((1-p)g_3(2at(1-p))+pg_3(2btp)\Big),
	\end{align*}
	which in combination with \eqref{eq:IntersectionMeasure} yields the result.
\end{proof}

\begin{proof}[Proof of Corollary \ref{cor:VarianceAsymptotic}]
	One can check that $g_{1}(x)\asymp-\log(x)$, $g_{2}(x) \asymp - \frac{1}{2} \log(x)$ and that $g_{3}(x) \asymp \log(x)$, as $x\to\infty$. The result then follows from Theorem \ref{Thm:Variances}.
\end{proof}

\section{Auxiliary computations}\label{sec:Aux}
In order to prove our results on the Mondrian pair- and cross-correlation functions, we will use results for a more general setting provided in \cite[Section 3]{ST-Plane}. The subsequent section gives the auxiliary computations necessary to adapt these results to our set-up of Mondrian tessellations with  weighted axis-parallel cutting directions.
We start with some general observations for rectangular STIT  tessellations that we will use repeatedly throughout the proofs. For a segment $ e $ of a STIT tessellation $Y_t$ on $\RR^2$ we introduce the measure
\begin{equation} \label{eq:DeltaE}
	\Delta^e := \sum_{x \in \text{Vertices}(e)} \delta_x,
\end{equation}
where we denote by \text{Vertices}(e) the set of all vertices in the relative interior of the maximal edge $e$. Furthermore, for  $e \in Y_t$ we define the measure
\begin{equation}\label{eq:LambdaE}
	(\Lambda[ \cdot \cap e])(A) := \Lambda([ A\cap e]),
\end{equation}
on $ \RR^2 $ (cf.\! \cite[Section 3.2]{ST-Plane}). Also, for bounded measurable functions $ f :\RR^2\rightarrow \RR $ with bounded support and $t>0$ we define the integrals
\begin{equation}\label{eq:I(n)1}
	\mathcal{I}^1(f(s);t) := \int_0^t f(s) \,  \dint s,
\end{equation}
and 
\begin{equation}\label{eq:I(n)2}
	\mathcal{I}^n(f(s);t) := \int_0^t \int_0^{s_1} \ldots \int_0^{s_{n-1}} f(s) \,  \dint s \,  \dint s_{n-1} \ldots \dint s_1 = \frac{1}{(n-1)!} \int_0^t (t-s)^{n-1}f(s) \dint s,
\end{equation}
for integers $n\geq 2$.
We now look at the measures defined in \eqref{eq:DeltaE} and \eqref{eq:LambdaE} in our set up.
Since we are dealing with rectangular tessellations, the endpoints of any segment need to coincide in one coordinate, i.e.,  segments parallel to $ E_1 $ have endpoints $ (\tau, \sigma) $ and $(\omega, \sigma )$, endpoints of segments parallel to $ E_2 $ are of the form  $ (\sigma,\tau ) $ and $(\sigma, \omega )$. For ease of notation we will write  $ \overline{(\tau \omega)}_\sigma $ for segements of the first and $ {_\sigma}{\overline{(\tau \omega)} }$ for  segments of the second kind. We now see that 
\begin{align}\label{eq:LambdaSegmentIntersection}
	\Lambda_p([\cdot \cap \overline{ (\tau,\omega)_\sigma}])=
	(1-p) \,\ell_{1}(\cdot \cap \overline{ (\tau,\omega)}_\sigma), 
	\quad \qquad 
	\Lambda_p ([\cdot \cap \overline{ {_\sigma}{(\tau,\omega)}} ])=p \,\ell_{1}(\cdot \cap \overline{ {_\sigma}{(\tau,\omega)}})
\end{align}
and
\begin{align} \label{eq:Delta}
	\Delta^{e} = \begin{cases}
		\delta_{(\tau ,\sigma)} +  \delta_{(\omega, \sigma)}  &\text{ for }\ e= \overline{(\tau \omega)}_\sigma,\\[0.2cm]
		\delta_{(\sigma, \tau)} +  \delta_{(\sigma, \omega)}  &\text{ for }\ e= {_\sigma}{\overline{(\tau \omega)} }.
	\end{cases}
\end{align}
We can now use \eqref{eq:LambdaSegmentIntersection} and \eqref{eq:Delta} to simplify the expression in \cite[Theorem 3]{ST-Plane}. To do so,
the lemma below gives simplified expressions for integrals over product measures where one or both of the measures is a Dirac measure or the $1$-dimensional Lebesgue measure of the intersection of a line segment with a subset of $ \RR^2 $.
\begin{lemma}\label{lem:IntegralExpressions}
	Let $ A, B \in \cB(\RR^2) $. 
	\begin{itemize}
		\item[(i)] For  any $ q \in [0,1 ]$ and $j \in \mathbb{N}$ we have
		\begin{align*}
			&\int_{\RR} \int_{\RR} \int_{\RR}(\delta_{(\tau,\sigma)}\otimes\ell_1(\cdot\cap  \overline{(\tau \vartheta)}_\sigma)(A \times B) \,  \mathcal{I}^j\big(s^2 \exp(-s q|\tau-\vartheta| );t\big)\ \dint \vartheta \,  \dint \tau \,  \dint \sigma\\
			&\qquad 	=\int_{\RR^2} \delta_{w}(A)\int_{\RR}\, \ell_1(B-w\cap  \overline{(0 z)}_0))\, \mathcal{I}^j\big(s^2 \exp(-s q|z| );t\big)   \,\dint z \, \dint w
			\intertext{and }
			& \int_{\RR} \int_{\RR} \int_{\RR}(\delta_{(\sigma, \tau)}\otimes\ell_1(\cdot\cap  {_\sigma}{\overline{(\tau \vartheta)}})(A \times B)  \,\mathcal{I}^j\big(s^2 \exp(-s q|\tau-\vartheta| );t\big)\ \dint \vartheta \,  \dint \tau \,  \dint \sigma\\
			&\qquad 	=\int_{\RR^2} \delta_{w}(A)\int_{\RR}\ell_1(B-w\cap  \overline{{_0}{(0z)}}))\, \mathcal{I}^j\big(s^2 \exp(-s q|z| );t\big)  \, \dint z \, \dint w.%
		\end{align*}
		\item[(ii)] For any $ q \in [0,1 ]$ and  $ j \in \mathbb{N} $  we have
		\begin{align*}
			&\int_{\RR} \int_{\RR} \int_{\RR}(\ell_1(\cdot\cap  \overline{(\tau \vartheta)}_\sigma)\otimes\ell_1(\cdot\cap  \overline{(\tau \vartheta)}_\sigma)(A \times B))  \,\mathcal{I}^j\big(s^2 \exp(-s q|\tau-\vartheta| );t\big)\ \dint \vartheta \,  \dint \tau \,  \dint \sigma\\
			& \quad 	 =   \int_{\RR^2} \delta_{w}(A) \int_{\RR} \Big(\int_{ \overline{(0z)}_0}  \int_{  \overline{(0z)}_0}    \delta_{y-x}(B-w)\, \dint x \, \dint y\Big)\,\mathcal{I}^j\big(s^2 \exp(-sq\vert z\vert) ;t\big)\,  \dint z \, \dint w
			\intertext{ and }
			&\int_{\RR} \int_{\RR} \int_{\RR}(\ell_1(\cdot\cap  {_\sigma}{\overline{(\tau \vartheta)}}) \otimes \ell_1(\cdot\cap  {_\sigma}{\overline{(\tau \vartheta)}}))(A \times B))  \,\mathcal{I}^j\big(s^2 \exp(-s q|\tau-\vartheta| );t\big)\ \dint \vartheta \,  \dint \tau \,  \dint \sigma\\
			& \quad 	 =   \int_{\RR^2} \delta_{w}(A)  \int_{\RR} \Big(\int_{ {_0}{\overline{(0z)}}}  \int_{  {_0}{\overline{(0z)}}}    \delta_{y-x}(\cdot)\, \dint x \, \dint y\Big)\,  \mathcal{I}^j\big(s^2 \exp(-sq\vert z\vert) ;t\big)\, \dint z \, \dint w.
		\end{align*}
		\item[(iii)] For any $ q \in [0,1] $ and $j \in \mathbb{N}$ we have
		\begin{align*}\
			&\int_{\RR} \int_{\RR} \int_{\RR}(\delta_{(\tau,\sigma)}\otimes\delta_{(\tau,\sigma)})(A \times B) \, \mathcal{I}^j\big(s^2 \exp(-s q|\tau-\vartheta| );t\big)\ \dint \vartheta \,  \dint \tau \,  \dint \sigma\\
			&\qquad=  \int_{\RR^2} \delta_{\omega}(A) \int_{\RR}  \delta_{\mathbf{0}}(B-w)\,\mathcal{I}^j\big(s^2 \exp(-sq\vert z\vert) ;t\big)\,  \dint z \, \dint  w,
			\intertext{where $\mathbf{0}:= (0,0)$. Furthermore,}
			\\
			&\int_{\RR} \int_{\RR} \int_{\RR}(\delta_{(\tau,\sigma)}\otimes\delta_{(\vartheta,\sigma)})(A \times B) \, \mathcal{I}^j\big(s^2 \exp(-s q|\tau-\vartheta| );t\big)\ \dint \vartheta \,  \dint \tau \,  \dint \sigma\\
			&\qquad=  \int_{\RR^2} \delta_{w}(A) \int_{\RR}  \delta_{(z,0)}(B-w)\,\mathcal{I}^j\big(s^2 \exp(-sq\vert z\vert) ;t\big)\,  \dint z \, \dint  w
			\intertext{ and }
			&\int_{\RR} \int_{\RR} \int_{\RR}(\delta_{(\sigma,\tau)}\otimes\delta_{(\sigma,\vartheta)})(A \times B) \, \mathcal{I}^j\big(s^2 \exp(-s q|\tau-\vartheta| );t\big)\ \dint \vartheta \,  \dint \tau \,  \dint \sigma\\
			&\qquad=  \int_{\RR^2} \delta_{w}(A) \int_{\RR}  \delta_{(0,z)}(B-w)\,\mathcal{I}^j\big(s^2 \exp(-sq\vert z\vert) ;t\big)\,  \dint z \, \dint  w.
		\end{align*}
	\end{itemize}
\end{lemma}
\begin{proof}
	In parts (i) and (ii) both cases can be shown with similar arguments, so that we will only prove one of them at a time. In part (iii) this holds true for the second and third claim, so that we will prove the first and second equality here. To prove (i), note that $ \ell_1(B\cap \overline{(\tau \vartheta)}_\sigma)=\ell_1(B-(\tau,\sigma)\cap \overline{(0 (\vartheta-\tau))}_0) $ for any $\tau, \vartheta, \sigma \in \RR$, so that 
	\begin{align*}
		& \int_{\RR} \int_{\RR} \int_{\RR}\delta_{(\tau,\sigma)}(A) \, \ell_1(B\cap  \overline{(\tau \vartheta)}_\sigma)\ \mathcal{I}^j\big(s^2 \exp(-sq\vert \tau-\vartheta \vert) ;t\big) \,\dint \vartheta \, \dint \tau \,  \dint \sigma\\[0.2cm]
		&= \int_{\RR} \int_{\RR}\delta_{(\tau,\sigma)}(A)\int_{\RR}\, \ell_1(B-(\tau,\sigma)\cap  \overline{(0 (\vartheta-\tau))}_0))\ \mathcal{I}^j\big(s^2 \exp(-sq| \tau-\vartheta| ) ;t\big) \, \dint \vartheta \, \dint \tau \, \dint \sigma.
	\end{align*}
	Substituting $ z=\vartheta-\tau $ we obtain that the above is equal to
	\begin{align*}
		\int_{\RR} \int_{\RR}\delta_{(\tau,\sigma)}(A)\int_{\RR}\, \ell_1(B-(\tau,\sigma)\cap  \overline{(0 z)}_0))\ \mathcal{I}^j\big(s^2 \exp(-s q|z| );t\big)\,   \dint z \, \dint \tau \, \dint \sigma.
	\end{align*}
	which shows (i) by putting $\omega=(\tau,\sigma)$.
	To prove part (ii), note that
	\begin{align*}
		&\int_{\RR}  \int_{\RR}  \int_{\RR} \ell_1(A \cap \overline{(\tau \vartheta)}_\sigma ) \ell_1(B \cap \overline{(\tau \vartheta)}_\sigma )\,\mathcal{I}^j\big(s^2 \exp(-sq\vert \tau-\vartheta \vert) ;t\big)\, \dint \tau \, \dint \vartheta \, \dint \sigma \\[0.2cm]
		&=  \int_{\RR}  \int_{\RR}  \int_{\RR} \Big(\int_{ \overline{(\tau \vartheta)}_\sigma}  \int_{ \overline{(\tau \vartheta)}_\sigma} \delta_x(A) \delta_y(B)\, \dint x \, \dint y\Big) \,\mathcal{I}^j\big(s^2 \exp(-sq\vert \tau-\vartheta \vert) ;t\big)\,\dint \tau \, \dint \vartheta \, \dint \sigma \\[0.2cm]
		&=  \int_{\RR}  \int_{\RR}  \int_{\RR} \Big(\int_{ \overline{(0(\vartheta-\tau))}_0}  \int_{  \overline{(0(\vartheta-\tau))}_0} \delta_x(A-(\tau,\sigma)) \delta_y(B-(\tau,\sigma))\, \dint x \, \dint y\Big) \\
		&\qquad \qquad \qquad \qquad\qquad \qquad\qquad \qquad\qquad \qquad\times\mathcal{I}^j\big(s^2 \exp(-sq\vert \tau-\vartheta \vert) ;t\big)\,\dint \tau \, \dint \vartheta \, \dint \sigma \\[0.2cm]
		&=  \int_{\RR^2}  \int_{\RR} \Big(\int_{ \overline{(0z)}_0}  \int_{  \overline{(0z)}_0}    \delta_{w+x}(A)\delta_{y}(B-w)\,\mathcal{I}^j\big(s^2 \exp(-sq\vert z\vert) ;t\big) \dint x \, \dint y\Big)\,  \dint z \, \dint w,
	\end{align*}
	where we substituted $z=\vartheta-\tau$ and put $ (\tau,\sigma)=w $. Substituting $\tilde{w}=w+x$ we end up with
	\begin{align*}
		\int_{\RR^2} & \int_{\RR} \Big(\int_{ \overline{(0z)}_0}  \int_{  \overline{(0z)}_0}    \delta_{w+x}(A)\delta_{y}(B-w)\, \dint x \, \dint y\Big)\,\mathcal{I}^j\big(s^2 \exp(-sq\vert z\vert) ;t\big)\,  \dint z \, \dint w\\[0.2cm]
		&=    \int_{\RR^2} \delta_{\tilde{w}}(A) \int_{\RR} \Big(\int_{ \overline{(0z)}_0}  \int_{  \overline{(0z)}_0}    \delta_{y-x}(B-\tilde{w})\, \dint x \, \dint y\Big)\,\mathcal{I}^j\big(s^2 \exp(-sq\vert z\vert) ;t\big)\,  \dint z \, \dint  \tilde{w},
	\end{align*}
	which proves the first equality in (ii). We now turn to the proof of the first part of item (iii).	Substituting $ z=\vartheta-\tau $ and setting $w:= (\tau, \sigma)$ we obtain	
	\begin{align*}
		&\int_{\RR} \int_{\RR} \int_{\RR} \, \delta_{(\tau ,\sigma)}(A)\delta_{(\tau ,\sigma)}(B) \,  \mathcal{I}^j\big(s^2 \exp(-sq| \tau-\vartheta| ) ;t\big)  \dint \vartheta\, \dint \tau \, \dint \sigma\\[0.2cm]
		& \quad =	\int_{\RR^2} \delta_{w}(A)\, \int_{\RR} \delta_{\mathbf{0}}(B - w)  \,\mathcal{I}^j\big(s^2 \exp(-sq|z| ) ;t\big) \, \dint z\,  \dint w.
	\end{align*}
	Using the same substitution as above we get
	\begin{align*}
		&\int_{\RR} \int_{\RR} \int_{\RR} \, \delta_{(\tau ,\sigma)}(A)\delta_{(\vartheta ,\sigma)}(B) \,  \mathcal{I}^j\big(s^2 \exp(-sq| \tau-\vartheta| ) ;t\big)  \, \dint \vartheta\,\dint \tau \, \dint \sigma\\[0.2cm]
		& \quad =	\int_{\RR^2} \delta_{w}(A)\, \int_{\RR} \delta_{(z,0)}(B - w)  \mathcal{I}^j\big(s^2 \exp(-sq|z| ) ;t\big)  \dint z\,  \dint w
	\end{align*}
	for the second term in (iii). Since the third expression can be handled in an analogous way, this completes the proof.
\end{proof}
Recall that in order to deduce pair- and cross-correlation functions of the vertex and the random edge length process we will observe a planar weighted Mondrian $Y_t(W)$ in a growing window $ rR_p $, where $ R_p=[0,1-p] \times [0,p] $. It is then straightforward to see that,
for $ z > 0 $, we have
\begin{align}\label{eq:DeltaRp}
	\delta_{(z,0)}(rR_p) = \mathbf{1}_{[0, (1-p)r]}(z),\qquad \text{ and } \qquad  \delta_{(0,z)}(rR_p) = \mathbf{1}_{[0, pr]}(z).
\end{align}

Furthermore, it holds that for  $z \leq 0$ we have $\ell_1(rR_p \cap \overline{(0z)}_0)=\ell_1(rR_p \cap \overline{{_0}{(0z)}})=0$ and
\begin{align}\label{eq:IntersectionRecLineSegE1}
	\ell_1(rR_p\cap \overline{(0z)}_0)=\begin{cases}
		z &\text{ for } z \in (0,r(1-p)],\\
		(1-p)r &\text{ for } z >r(1-p)
	\end{cases}
\end{align}
as well as
\begin{align}\label{eq:IntersectionRecLineSegE2}
	\ell_1(rR_p\cap \overline{{_0}{(0z)}})=\begin{cases}
		z &\text{ for } z \in (0,rp],\\
		pr &\text{ for } z >rp.
	\end{cases}
\end{align}
\begin{lemma}\label{lem:DoubleIntegral0zDirac}
	Take $ R_{r,p}=[0,u_{1}(r,p)]\times [0, u_2(r,p)] $ with non-negative functions $u_1,u_2:(0,\infty)\times [0,1]\rightarrow (0,\infty)$. It then  holds that 
	\begin{align*}\int_{\overline{{(0z)}}_0 } \int_{ \overline{{(0z)}}_0 }\delta_{y-x}( R_{r,p})\, \dint x \, \dint y \ = \
		\begin{cases}
			zu_1(r,p)- \frac{u_1(r,p)^2}{2} & \text{ for } |z| \geq u_1(r,p),\\[0.2cm]
			\frac{z^2}{2} & \text{ for } 0<|z| < u_1(r,p),
		\end{cases}
		\intertext{ and }
		\int_{   \overline{{_0}{(0z)}}}  \int_{  \overline{{_0}{(0z)}} }\delta_{y-x}( R_{r,p})\, \dint x \, \dint y \ = \
		\begin{cases}
			zu_2(r,p)- \frac{u_2(r,p)^2}{2} & \text{ for } |z| \geq u_2(r,p),\\[0.2cm]
			\frac{z^2}{2} & \text{ for } 0<|z| < u_2(r,p).
		\end{cases}
	\end{align*}
	
\end{lemma}
\begin{proof} We start by proving the second claim. 
	Note that for $x,y \in \overline{{_0}{(0z)}}$
	\[
	\delta_{y-x}(R_{r,p})=\delta_{(0,y)}(R_{r,p} +(0,x))=1 \qquad \text{if and only if} \qquad  y \in [x,x+u_2(r,p)].
	\] Assume $ z\geq 0 $, the case $ z<0 $ follows an analogous line of argumentation. For $ z \geq  u_2(r,p) $ we have
	\begin{align*}
		\int_{ \overline{{_0}{(0z)}}}  \int_{  \overline{{_0}{(0z)}}} \delta_{y-x}(R_{r,p})\, \dint y \, \dint x& = \int_{ 0}^z  \int_{ x}^{(x +u_2(r,p)) \wedge z } 1 \ \dint y \, \dint x\\[0.2cm]
		&= \int_{ 0}^{z-u_2(r,p)}  \int_{ x}^{x +u_2(r,p)} 1 \ \dint y \, \dint x+\int_{ z-u_2(r,p)}^{z}  \int_{ x}^{z} 1\ \dint y \, \dint x\\
		&= u_2(r,p)(z-u_2(r,p))+ \int_{z-u_2(r,p)}^z (z-x) \, \dint x\\[0.2cm]
		&= u_2(r,p)(z-u_2(r,p))+\frac{u_2(rp)^2}{2}=zu_2(r,p)-\frac{(u_2(r,p))^2}{2}
	\end{align*}
	and for $z \in [0,u_2(r,p))$ we get
	\begin{align*}
		\int_{ \overline{{_0}{(0z)}}}  \int_{  \overline{{_0}{(0z)}}} \delta_{y-x}(R_{r,p})\, \dint y \, \dint x& = \int_0^z \int_x^z 1\, \dint x \, \dint y=\int_{0}^z (z-x) \, \dint x= \frac{z^2}{2}.
	\end{align*}
	The proof of the first claim works in an analogous way. 
\end{proof}

To deal with the Mondrian pair- and cross-correlation functions we will need to differentiate terms of the form  
$$ u(r,p)^k \int_0^{u(r,p)} f(z) \ \mathcal{I}^j\big(s^2 \exp(-sqz) ;t\big) \dint z \qquad \textup{  and  } \qquad u(r,p)^k \int_{u(r,p)}^{\infty} f(z)\  \mathcal{I}^j\big(s^2 \exp(-sqz) ;t\big) \dint z $$
with respect to $ r $ for a differentiable function $u: \RR\times [0,1] \rightarrow \RR^+$ and $k \in \NN$. Note that  for $ k \geq 1 $ we have
\begin{align}\label{eq:DerivativeIntegralI}
	&\frac{\dint }{\dint r}\Bigg[ u(r,p)^k \int_0^{u(r,p)} f(z)  \,\mathcal{I}^j\big(s^2 \exp(-sqz) ;t\big)\, \dint z\Bigg]\notag\\[0.2cm]
	&\qquad= k \, u(r,p)^{k-1}\, u'(r,p)\, \int_0^{u(r,p)} f(z) \, \mathcal{I}^j\big(s^2 \exp(-sqz) ;t\big)\, \dint z\notag\\[0.2cm]
	&\qquad \quad+  u(r,p)^k\, u'(r,p)\, f((u(r,p))) \ \mathcal{I}^j\big(s^2 \exp(-squ(r,p)) ;t\big)
\end{align}
and
\begin{align}\label{eq:DerivativeIntegralII}
	&\frac{\dint }{\dint r}\Bigg[ u(r,p)^k \int_{u(r,p)}^\infty f(z)  \,\mathcal{I}^j\big(s^2 \exp(-sqz) ;t\big) \,\dint z\Bigg]\notag\\[0.2cm]
	&\qquad=k\, u(r,p)^{k-1} \,u'(r,p)\, \int_{u(r,p)}^\infty f(z) \, \mathcal{I}^j\big(s^2 \exp(-sqz) ;t\big) \,\dint z\notag\\[0.2cm]
	&\qquad \quad - u(r,p)^k \,u'(r,p)\, f((u(r,p)))\ \mathcal{I}^j\big(s^2 \exp(-squ(r,p)) ;t\big),
\end{align}
where $ u' $ denotes the partial derivative of $u$ with respect to the $ r $-coordinate.
Combining Lemma~\ref{lem:DoubleIntegral0zDirac} with	Equations~\eqref{eq:DerivativeIntegralI} we obtain
\begin{align}\label{DerivativeT3}
	&\frac{\dint}{\dint r}\Bigg[ \int_{\RR} \Big(\int_{\overline{{_0}{(0z)}}} \int_{ \overline{{_0}{(0z)}}} \delta_{y-x}(R_{r,p})\, \dint y \, \dint x\Big) \,\mathcal{I}^j\big(s^2 \exp(-sq z) ;t\big)\,\dint z\Bigg]\notag \\
	&\quad =\frac{\dint}{\dint r}\Bigg[ \int_{0}^{u_{2}(r,p)}z^2\,\mathcal{I}^j\big(s^2 \exp(-sq z) ;t\big)\,   \dint z+  2\int_{u_{2}(r,p)}^\infty \Big(zu_{2}(r,p)-\frac{(u_{2}(r,p))^2}{2}\Big)\,\mathcal{I}^j\big(s^2 \exp(-sqz) ;t\big)\,   \dint z\Bigg]\notag\\
	\notag\\
	&\quad=u_{2}(r,p)^2\,  u_{2}'(r,p)\ \mathcal{I}^j\big(s^2 \exp(-squ_{2}(r,p)) ;t\big)+2 u_{2}'(r,p)\, \int_{u_{2}(r,p)}^\infty \, z\,\mathcal{I}^j\big(s^2 \exp(-sqz) ;t\big) \,\dint z\notag\\[0.2cm]
	&\qquad-2u_{2}(r,p)^2u_{2}'(r,p)\, \mathcal{I}^j\big(s^2 \exp(-sq u_{2}(r,p))  ;t\big) \notag\\
	& \qquad -2u_2(r,p)\, u'_2(r,p) \int_{u_{2}(r,p)}^\infty \, \mathcal{I}^j\big(s^2 \exp(-s(1-p)z) ;t\big) \,\dint z\notag\\[0.2cm]
	&\qquad +u_{2}(r,p)^2 u_{2}'(r,p)\, \mathcal{I}^j\big(s^2 \exp(-sq u_{2}(r,p))  ;t\big)\notag\\[0.2cm]
	&\quad=2 u_{2}'(r,p)\, \int_{u_{2}(r,p)}^\infty \, z\,\mathcal{I}^j\big(s^2 \exp(-sqz) ;t\big)\, \dint z-2u_2(r,p)\, u'_2(r,p) \int_{u_{2}(r,p)}^\infty \, \mathcal{I}^j\big(s^2 \exp(-s(1-p)z) ;t\big)\, \dint z\notag\\[0.2cm]
	&\quad=2 u_{2}'(r,p)\, \int_{u_{2}(r,p)}^\infty \Big( z- u_{2}(r,p))\Big)\ \mathcal{I}^j\big(s^2 \exp(-sqz) ;t\big)\, \dint z.
\end{align}
Proceeding analogously with the first integral in Lemma~\ref{lem:DoubleIntegral0zDirac} we obtain
	\begin{align}\label{DerivativeT2}
		&\frac{\dint}{\dint r}\Bigg[ \int_{\RR} \Big(\int_{\overline{{(0z)_0}}}\int_{\overline{{(0z)_0}}}  \delta_{y-x}(R_{r,p})\, \dint y \, \dint x\Big) \,\mathcal{I}^j\big(s^2 \exp(-sq z) ;t\big)\,\dint z\Bigg]\notag \\
		&\quad=2 u_{1}'(r,p)\, \int_{u_{1}(r,p)}^\infty \Big( z- u_{1}(r,p))\Big)\ \mathcal{I}^j\big(s^2 \exp(-sqz) ;t\big)\, \dint z.
	\end{align}
When calculating pair- and cross-correlation functions, we will come across several integrals of functions $f(z) = z^i \mathcal{I}^j\big(s^2 \exp(-sz) ;t\big)$ for $i,j \in \{1,2,3\}$. We collect  their exact values in the following lemma for easy reference. 
\begin{lemma} \label{lem:ValueIntegralCalc}
	Let $q \in [0,1], t\in [0, \infty), s \in [0,t], i,j \in \{1,2,3\}$ and $\mathcal{I}^j$ as defined in \eqref{eq:I(n)1} and \eqref{eq:I(n)2}. Then it holds that
	\allowdisplaybreaks
	
	\begin{align*}
		& \mathcal{I}^j(s^2\exp(-sq);t)= q^{-(j+2)}\bigg(\sum_{r=0}^{j-1}\Big[(-1)^{r+j+1}(qt)^r
		\frac{(j+1-r)!}{(j-1-r)!\,r!}\Big]-\exp(-qt)\big(qt(qt+2j)+\frac{(j+1)!}{(j-1)!}\big)\bigg)
		\intertext{ and , for $y \ge 0$,}
		&\int_{y}^\infty z^k\,  \mathcal{I}^j(s^2\exp(-sqz);t)\, \dint z\\ 
		&\qquad =y^{-(j+1-k)}q^{-(2+j)}\bigg( (-1)^{j+1} \sum_{r=0}^{j-1}(-1)^{r}\frac{k+j-r}{r!}(qyt)^r+(-1)^j\exp(-tqy)\Big((k+j)+qty\Big)\bigg).
	\end{align*}
\end{lemma}

\allowdisplaybreaks


\section{Surface covariance measure and edge pair correlations} \label{sec:EdgePair}

This section aims to show Theorem \ref{thm:EdgePair}. For a weighted planar Mondrian tessellation $Y_t$ on $\RR^2$ with driving measure $\Lambda_p$ the covariance measure $\Cov(\mathcal{E}_{t})$ of its random edge length measure $\mathcal{E}_{t}$ is defined as the measure on $\RR^2 \times \RR^2$ given by the relation
\begin{align*}
	\int_{(\RR^2)^2} (f \otimes g)(x) \, \Cov(\mathcal{E}_{t})(\dint x) = \Cov\Big(\Sigma_{J^f}(Y_t),\Sigma_{J^g}(Y_t)\Big) ,
\end{align*}
for bounded and measurable $f,g:\RR^2 \to \RR$ with bounded support, where $J^f$ is the edge functional

$$J^f(e) := \int\limits_e f(x) \,\dint x.$$
Slightly adapting the calculations in \cite[Section 3.3]{ST-Plane} and using the martingale arguments from \cite[Equation (3)]{ST-Plane} to calculate expectations for the functional $A_{J^f J^g}$ as in \eqref{eq:FuncAllgA}, we get 
\begin{align*}
	\int_{(\RR^2)^2} (f \otimes g)(x)  \, \Cov(\mathcal{E}_{t})(\dint x) 
	& = \int\limits_0^t \EE A_{J^f J^g}(Y_s) \,\dint s \\
	&= \frac{1}{2} \int\limits_{\mathcal{S}(\RR^2)} J^f(e)J^g(e) \int\limits_0^t s^2e^{s\Lambda_{p}[e]}ds  \ll (\Lambda_p \times \Lambda_p) \cap \Lambda_p \gg (\dint e).
\end{align*}
By the definition of $J^f$, $J^g$ and the segment intersection measure from \eqref{eq:DefSegIntersecMeasure}, we obtain the following form for the covariance measure:
\begin{align}\label{eq:EdgeCov}
	\nonumber \Cov&(\mathcal{E}_{t}) ( \cdot \times \cdot) = \frac{1}{2} \int\limits_{\mathcal{S}(\RR^2)} \ell_1(\cdot \cap e) \otimes \ell_1(\cdot \cap e) ( \cdot \times \cdot )\, \mathcal{I}^1(s^2e^{-s\Lambda_p[e]};t) \, \ll(\Lambda_p \times \Lambda_p) \cap \Lambda_p \gg (\dint e)\\
	\nonumber=& \frac{1}{2} \int\limits_{[\RR^2]} \int\limits_{[L]} \int\limits_{[L]} \ell_1(\cdot \cap L(L_1,L_2)) \otimes \ell_1(\cdot \cap L(L_1,L_2))( \cdot \times \cdot )\,  \mathcal{I}^1 (s^2 e^{-s \Lambda_p([L(L_1,L_2)])};t) \\
	\nonumber & \qquad \qquad \qquad \qquad  \qquad \qquad \qquad \qquad  \qquad  \qquad \qquad  \times \,  \Lambda_p(\dint L_1) \, \Lambda_p(\dint L_2)\Lambda_p(\dint L) \, \\
	%
	\nonumber =& \frac{1}{2} p(1-p)^2 \int\limits_\RR \int\limits_\RR \int\limits_\RR \ell_1(\cdot \cap \overline{(\tau \vartheta)}_\sigma) \otimes \ell_1(\cdot \cap \overline{(\tau \vartheta)}_\sigma) \, ( \cdot \times \cdot )\, \mathcal{I}^1 (s^2 e^{-s(1-p)|\tau-\vartheta|};t) \,\dint \tau \,\dint \vartheta\,\dint \sigma\\
	&\quad +  \frac{1}{2} p^2(1-p) \int\limits_\RR \int\limits_\RR \int\limits_\RR \ell_1(\cdot \cap {_\sigma}{\overline{(\tau \vartheta)} }) \otimes \ell_1(\cdot \cap {_\sigma}{\overline{(\tau \vartheta)} }) \, ( \cdot \times \cdot )\, \mathcal{I}^1 (s^2 e^{-sp|\tau-\vartheta|};t) \,\dint \tau \,\dint \vartheta\,\dint \sigma.
\end{align}
Having established the covariance measure of the edge process $\mathcal{E}_{t}$, we now aim at giving the corresponding pair-correlation function $g_\mathcal{E}(r)$. In a first step towards this we need to establish the reduced covariance measure $\widehat{\Cov}(\mathcal{E}_{t})$  defined by the relation
\begin{align*}
\Cov(\mathcal{E}_{t})(A\times B)= \int_{A} \int_{B-x}\widehat{\Cov}(\mathcal{E}_{t})(\dint y)\ \ell_2(\dint x)
\end{align*}
for a measurable product $A\times B\subset\RR^2\times\RR^2$ (cf. \cite[Corollary 8.1.III]{DaleyVerejones}). 
We now examine the first of the two integral summands in \eqref{eq:EdgeCov}. Using Lemma~\ref{lem:IntegralExpressions}(ii) we see that
\begin{align}\label{eq:DoubleIntegralLebesgue}
	&\, \,\frac{1}{2} p(1-p)^2 \int\limits_\RR \int\limits_\RR \int\limits_\RR \ell_1(A \cap \overline{(\tau \vartheta)}_\sigma) \, \ell_1(B \cap \overline{(\tau \vartheta)}_\sigma) \, \mathcal{I}^1 (s^2 e^{-s(1-p)|\tau-\vartheta|};t)\, \dint \tau\,\dint \vartheta\,\dint \sigma \notag\\
		&= \frac{1}{2} p(1-p)^2 \int\limits_{\RR^2} \delta_{w}(A) \int\limits_\RR \int\limits_{\overline{(0 z)}_0} \int\limits_{\overline{(0 z)}_0} \delta_{y-x}(B-\tilde w) \, \dint x \, \dint y \, \mathcal{I}^1 (s^2 e^{-s(1-p)|z|};t) \, \dint z \, \dint w.
\end{align}

Proceeding analogously with the second summand in \eqref{eq:EdgeCov} and using the diagonal shift argument from \cite[Corollary 8.1.III]{DaleyVerejones}, we get the reduced covariance measure $  \widehat{\Cov}(\mathcal{E}_{t})$ on $\RR^2$:
\begin{align*}
	\widehat{\Cov}(\mathcal{E}_{t})( \, \cdot \, ) &= \frac{1}{2} p(1-p)^2 \int\limits_\RR \int\limits_{\overline{(0 z)}_0} \int\limits_{\overline{(0 z)}_0} \delta_{y-x}(\cdot ) \, \dint x \, \dint y \, \mathcal{I}^1 (s^2 e^{-s(1-p)|z|};t) \, \dint z\\
	&\quad +  \frac{1}{2} p^2(1-p) \int\limits_\RR \int\limits_{{_0}\overline{(0 z)}} \int\limits_{{_0}\overline{(0 z)}} \delta_{y-x}(\cdot ) \, \dint x \, \dint y \, \mathcal{I}^1 (s^2 e^{-s(1-p)|z|};t) \, \dint z.
\end{align*}
Noting that the intensity of the random measure $\mathcal{E}_{t}$ is just $t$, see \cite[Equation (8)]{ST-Plane}, we apply Equation (8.1.6) in \cite{DaleyVerejones} to see that the corresponding reduced second moment measure $\widehat{\mathcal{K}}(\mathcal{E}_{t})$ is
\begin{align*}
\widehat{\mathcal{K}}(\mathcal{E}_{t})( \cdot) &= \widehat{\Cov}(\mathcal{E}_{t})( \, \cdot \, ) + t^2 \ell_2(\cdot).
\end{align*}
While the classical Ripley's K-function would be $t^{-2}$ times the $\widehat{\mathcal{K}}(\mathcal{E}_{t})$-measure of a disc of radius $r>0$, we define our Mondrian analogue as
\begin{align*} 
	K_{\mathcal{E}}(r):=\frac{1}{t^2} \, \widehat{\mathcal{K}}(\mathcal{E}_{t})(R_{r,p}), 
\end{align*}
where $R_{r,p}:=rR_p$ with $R_p:=[0,1-p] \times [0, p]$ as before. 
Calculating $K_{\mathcal{E}}(r)$ explicitly via Lemma \ref{lem:DoubleIntegral0zDirac} yields %
\allowdisplaybreaks
\begin{align*}
	K_{\mathcal{E}}(r) &= r^2p(1-p) +\frac{p(1-p)^2}{2t^2}  \int\limits_\RR \int\limits_{\overline{(0 z)}_0} \int\limits_{\overline{(0 z)}_0} \delta_{y-x}(R_{r,p}  ) \, \dint x \, \dint y \, \mathcal{I}^1 (s^2 e^{-s(1-p)|z|};t) \, \dint z\\
	&\quad +  \frac{p^2(1-p)}{2t^2}  \int\limits_\RR \int\limits_{{_0}\overline{(0 z)}} \int\limits_{{_0}\overline{(0 z)}} \delta_{y-x}(R_{r,p} ) \, \dint x \, \dint y \, \mathcal{I}^1 (s^2 e^{-s(1-p)|z|};t) \, \dint z.
\end{align*}

With \eqref{DerivativeT3} and \eqref{DerivativeT2} 
this gives
\begin{eqnarray*}
	\frac{d}{dr}K_{\mathcal{E}}(r) &=& 2rp(1-p) + \frac{p(1-p)^{3}}{t^2} \int\limits_{(1-p)r}^{\infty} (z - r(1-p)) \, \mathcal{I}^1 (s^2 e^{-s(1-p)z};t) \, \dint z\\
	&& + \frac{(1-p)p^3}{t^2}  \int\limits_{rp}^{\infty}( z - rp ) \,  \mathcal{I}^1 (s^2 e^{-spz};t) \,  \dint z.
\end{eqnarray*}
The definition of the function $g_\mathcal{E}$ as given in Section~\ref{sec:CorrFct} together with the integral calculations from Lemma~\ref{lem:ValueIntegralCalc}
then gives the function in the statement of Theorem \ref{thm:EdgePair}.\qed


\section{Edge-vertex correlations for Mondrian tessellations} \label{sec:CrossCorr}

We now turn to the proof of Theorem \ref{thm:CrossCorr} by deriving the concrete cross-covariance measure of the vertex process $\mathcal{V}_{t}$ and the random edge length measure $\mathcal{E}_{t}$ of a weighted Mondrian tessellation $Y_t$ on $\RR^2$ with driving measure $\Lambda_p$.Using the general formula for this cross-covariance measure from \cite[Theorem 3]{ST-Plane} for the driving measure $\Lambda_p$ together with the definition of the segment intersection measure in \eqref{eq:DefSegIntersecMeasure} yields
\begin{align}\label{eq:EdgeVertexCov}
	\nonumber &\Cov(  \mathcal{V}_{t},\mathcal{E}_{t})( \cdot \times \cdot )\\
	\nonumber &=
	\int_{[W]} \int_{[L \cap W]} \int_{[L \cap W]}\Bigg[\frac{1}{2}\Delta^{L(L_1,L_2)}\otimes \ell_1(\cdot \cap L(L_1,L_2) ) \ ( \cdot \times \cdot )\, \mathcal{I}^2\big(s^2 \exp(-s\Lambda_p([L(L_1,L_2)]);t\big)\\
	\\
	\nonumber &  \qquad \qquad  + \Lambda_p\big([\cdot \cap L(L_1,L_2)]\big)\otimes \ell_1(\cdot \cap L(L_1,L_2) )\,( \cdot \times \cdot )\, \mathcal{I}^2\big(s^2 \exp(-s\Lambda_p([L(L_1,L_2)]));t\big)\Bigg]\\
	& \qquad \qquad \qquad \qquad  \qquad \qquad \qquad \qquad  \qquad  \qquad   \qquad  \qquad \qquad   \times \, \Lambda_p(\dint L_1) \, \Lambda_p(\dint L_2) \, \Lambda_p(\dint L).
\end{align}
With the definition of the measure $\Lambda_p$ given in \eqref{def:Lambda} and Equations~\eqref{eq:SegmentIntersectionMeasure}, \eqref{eq:LambdaSegmentIntersection} and \eqref{eq:Delta}  yield that
\begin{align}\label{eq:CovEV}
	&  \Cov(\mathcal{V}_{t}, \mathcal{E}_{t}) ( \cdot \times \cdot ) \notag\\&=
	p (1-p)^2\int_{\RR} \int_{\RR} \int_{\RR}\Big[\frac{1}{2} \big(\delta_{(\tau,\sigma)}+ \delta_{(\vartheta, \sigma)}\big)\otimes\ell_1(\cdot\cap \overline{(\tau \vartheta)}_\sigma)\ ( \cdot \times \cdot )\,  \mathcal{I}^1\big(s^2 \exp(-s(1-p)\vert \tau-\vartheta \vert) ;t\big)\notag\\
	&\qquad\qquad \qquad  + (1-p)\Big(\ell_1(\cdot \cap \overline{(\tau \vartheta)}_\sigma )\otimes \ell_1(\cdot \cap \overline{(\tau \vartheta)}_\sigma )\Big)\,( \cdot \times \cdot )\,  \mathcal{I}^2\big(s^2 \exp(-s(1-p)\vert \tau-\vartheta \vert) ;t\big)\Big] \,\dint \tau \, \dint \vartheta \, \dint \sigma
	\notag\\
	&\quad +  p^2 (1-p)\int_{\RR} \int_{\RR} \int_{\RR}\Big[\frac{1}{2}\big(\delta_{(\sigma,\tau)}+ \delta_{(\sigma, \vartheta)}\big)\otimes \ell_1(\cdot\cap\overline{{_\sigma}(\tau \vartheta)}) \ ( \cdot \times \cdot )\,  \mathcal{I}^1\big(s^2 \exp(-sp\vert \tau-\vartheta \vert) ;t\big)\notag\\
	&\qquad\qquad \qquad  +p \Big(\ell_1(\cdot \cap {_\sigma}{\overline{(\tau \vartheta)}} ) \otimes \ell_1(\cdot \cap {_\sigma}{\overline{(\tau \vartheta)}} )\Big)\, ( \cdot \times \cdot )\, \mathcal{I}^2\big(s^2 \exp(-sp\vert \tau-\vartheta \vert) ;t\big)\Big]\, \dint \tau \, \dint \vartheta \, \dint \sigma.
\end{align}
Let $ A, B \in \cB(\RR^2) $.
We can now use Lemma~\ref{lem:IntegralExpressions}(i) to deal with the first summand in both terms, applied to the product set $A\times B$, to obtain
\begin{align}
	&\frac{ (1-p)^2}{2}\int_{\RR} \int_{\RR} \int_{\RR}\big(\delta_{(\tau,\sigma)}+ \delta_{(\vartheta, \sigma)}\big)\otimes\ell_1(\cdot\cap \overline{(\tau \vartheta)}_\sigma)(A\times B)\ \mathcal{I}^1\big(s^2 \exp(-s(1-p)\vert \tau-\vartheta \vert) ;t\big)\, \dint \tau \, \dint \vartheta \, \dint \sigma\notag\\
	\notag \\
	&\quad=p (1-p)^2\int_{\RR^2} \delta_{w}(A)\int_{\RR}\, \ell_1(B-w\cap  \overline{(0 z)}_0)\mathcal{I}^1\big(s^2 \exp(-s (1-p)|z| );t\big)\dint z \, \dint w \label{eq:DiracLebesgue1}\\
	\intertext{and} \notag
	&  \frac{p^2 (1-p)}{2}\int_{\RR} \int_{\RR} \int_{\RR}\Big(\big(\delta_{(\sigma,\tau)}+ \delta_{(\sigma, \vartheta)}\big)\otimes \ell_1(\cdot\cap\overline{{_\sigma}(\tau \vartheta)})\Big)(A\times B) \ \mathcal{I}^1\big(s^2 \exp(-sp\vert \tau-\vartheta \vert) ;t\big) \,\dint \tau \, \dint \vartheta \, \dint \sigma\notag\\
	\notag\\
	&\quad =p^2 (1-p)\int_{\RR^2} \delta_{w}(A)\int_{\RR}\, \ell_1(B-w\cap  \overline{{_0}{(0 z)}}))\ \mathcal{I}^1\big(s^2 \exp(-s p|z| );t\big) \, \dint z \, \dint w. \label{eq:DiracLebesgue2}
\end{align}
 The second summand in each of the terms in \eqref{eq:CovEV} can be dealt with using Lemma~\ref{lem:IntegralExpressions}(ii), see also Equation~\eqref{eq:DoubleIntegralLebesgue}.
As in the previous section, we want to proceed by giving the reduced covariance measure via the diagonal-shift argument in the sense of \cite[Corollary 8.1.III]{DaleyVerejones}. Plugging the terms we just deduced into \eqref{eq:CovEV}, we end up with the covariance measure
\begin{align*}
	\Cov( \mathcal{V}_{t},\mathcal{E}_{t})(A\times B)= \int_{A} \int_{B-x}\widehat{\Cov}_{\cV,\cE}(\dint y)\ \ell_2(\dint x),
\end{align*}
where the reduced cross-covariance measure is given by
\begin{align}\label{eq:ReducedCrossCovariance}
	\widehat{\Cov}_{\cV,\cE}( \, \cdot \, )= p(1-p)&\Bigg( (1-p)\int_{\RR} \ell_1(\cdot\cap\,\overline{(0z)}_0) \,\mathcal{I}^1\big(s^2 \exp(-s(1-p)\vert z \vert ;t\big) \,\dint z\notag\\
	&\quad+ p\int_{\RR} \ell_1(\cdot\cap  {_0}{\overline{(0z)}})\, \mathcal{I}^1\big(s^2 \exp(-sp\vert z \vert ;t\big)\, \dint z\notag\\
	&\quad+ (1-p)^2 \int_{\RR} \Big(\int_{ \overline{(0z)}_0}  \int_{  \overline{(0z)}_0}    \delta_{y-x}(\cdot)\, \dint x \, \dint y\Big)\,\mathcal{I}^2\big(s^2 \exp(-s(1-p)\vert z\vert) ;t\big)\,  \dint z \notag \\
	&\quad+ p^2\int_{\RR} \Big(\int_{ {_0}{\overline{(0z)}}}  \int_{  {_0}{\overline{(0z)}}}    \delta_{y-x}(\cdot)\, \dint x \, \dint y\Big)\,  \mathcal{I}^2\big(s^2 \exp(-sp\vert z\vert) ;t\big)\, \dint z \Bigg).
\end{align}
Thus, the reduced second cross moment measure has the form
\begin{align*}
		\widehat{\mathcal{K}_{\cV,\cE}}( \, \cdot \, )&=	\widehat{\Cov}_{\cV,\cE}( \, \cdot \, )+ \lambda_{\cV} \lambda_{\cE} \ell_2 (\cdot).
\end{align*}

The Mondrian cross K-function is now defined as
\begin{align}\label{eq:Cross-KFunction}
	K_{\cV,\cE}(r)=\frac{1}{\lambda_{\cV} \lambda_{\cE}} \widehat{\mathcal{K}}_{\cV,\cE}(R_{r,p}), \qquad r>0,
\end{align}
where $\lambda_\cV=t^2p(1-p) $ is the intensity of $\cV_{t}$, $\lambda_{\cE}=t$ that of $\cE_{t}$ and $R_{r,p}=rR_p=[0,r(1-p)]\times [0,rp]$.
Using Equation~\eqref{eq:ReducedCrossCovariance} we see that
\begin{align}\label{eq:K12}
K_{\cV,\cE}(R_{r,p})	&=  r^2 p(1-p)+\frac{p(1-p)}{t^3} \,\Big[ T_1(r)+(1-p)^2\, T_2(r)+p^2\, T_3(r)\Big].
\end{align}

Using Equations~\eqref{eq:IntersectionRecLineSegE1} and \eqref{eq:IntersectionRecLineSegE2} 
we can simplify the the term $T_1$ 
in \eqref{eq:K12}, and then use  \eqref{eq:DerivativeIntegralI} and \eqref{eq:DerivativeIntegralII} to conclude that
\begin{align*}
	\frac{\dint T_1(r)}{\dint r}
	&=\frac{\dint }{\dint r}\Bigg[(1-p)\Bigg(\int_0^{r(1-p)}z \, \mathcal{I}^1\big(s^2 \exp(-s(1-p)z );t\big) \,\dint z\\
	&  \qquad + r(1-p)\int_{r(1-p)}^{\infty}\, \mathcal{I}^1\big(s^2 \exp(-s(1-p)z) ;t\big) \,\dint z\Bigg)\\
	& \qquad  + p\Bigg(\int_0^{rp}z \, \mathcal{I}^1\big(s^2 \exp(-sp z)  ;t\big)\, \dint z+ rp\int_{rp}^{\infty}\, \mathcal{I}^1\big(s^2 \exp(-sp z)  ;t\big)\, \dint z\Bigg)\Bigg]\\
	&=(1-p)\Big( r(1-p)^2 \mathcal{I}^1\big(s^2 \exp(-s(1-p)^2r) ;t\big)+ (1-p) \int_{r(1-p)}^\infty \, \mathcal{I}^1\big(s^2 \exp(-s(1-p)z) ;t\big) \,\dint z\\
	&\qquad-r(1-p)^2\mathcal{I}^1\big(s^2 \exp(-s r(1-p)^2  ;t\big)\Big)+ p\Big( rp^2 \mathcal{I}^1\big(s^2 \exp(-srp^2 ) ;t\big)\\
	&\qquad+p \int_{rp}^\infty \, \mathcal{I}^1\big(s^2 \exp(-spz) ;t\big)\,\dint z -rp^2\mathcal{I}^1\big(s^2 \exp(-srp^2) ;t\big)\Big)\\
	&= (1-p)^2 \int_{r(1-p)}^\infty \, \mathcal{I}^1\big(s^2 \exp(-s(1-p)z) ;t\big)\,\dint z+ p^2 \int_{rp}^\infty \, \mathcal{I}^1\big(s^2 \exp(-spz) ;t\big)\,\dint z.
\end{align*}
For the second and third term we can evoke Equations~\eqref{DerivativeT3} and \eqref{DerivativeT2} to see that 
\begin{align*}
	\frac{\dint T_2(r)}{\dint r}&=\frac{\dint}{\dint r}\Bigg[ \int_{\RR} \Big(\int_{ \overline{(0z)}_0}  \int_{  \overline{(0z)}_0}    \delta_{y-x}(R_{r,p})\, \dint x \, \dint y\Big)\,\mathcal{I}^2\big(s^2 \exp(-s(1-p)\vert z\vert) ;t\big)\,  \dint z \Bigg]\\
	& =2(1-p) \int_{r(1-p)}^\infty \big(z-r(1-p)\big) \, \mathcal{I}^2\big(s^2 \exp(-s(1-p)z) ;t\big)\, \dint z
	\intertext{and}
	\frac{\dint T_3(r)}{\dint r}&=\frac{\dint}{\dint r}\Bigg[ \int_{\RR} \Big(\int_{ {_0}{\overline{(0z)}}}  \int_{  {_0}{\overline{(0z)}}}    \delta_{y-x}(R_{r,p})\, \dint x \, \dint y\Big)\,  \mathcal{I}^2\big(s^2 \exp(-sp\vert z\vert) ;t\big)\, \dint z \Bigg]\\
	&=2p \int_{rp}^\infty \big(z-rp\big) \, \mathcal{I}^2\big(s^2 \exp(-spz) ;t\big)\, \dint z.
\end{align*}

By plugging in the expressions for the respective integrals given in Lemma~\ref{lem:ValueIntegralCalc},  recalling Equation \eqref{eq:Cross-KFunction} we finally obtain the formula for $g_{\cV,\cE}(r)$ after simplification of the resulting expression.\qed


\section{Vertex Covariance Measure and Vertex Pair-Correlations} \label{sec:VertexPair}

In this section we derive Theorem \ref{thm:VertexPair}. As before we denote by $\mathcal{V}_{t}$ the vertex point process of a  weighted Mondrian tessellation $Y_t$ with driving measures $\Lambda_p$. As in the previous sections \cite[Theorem 2]{ST-Plane} provides us with a general formula for its covariance measure. Specialization to $\Lambda_p$ gives
\newpage
\begin{eqnarray} \label{eq:VertexPairCovMeasure}
	\nonumber &&\Cov(\mathcal{V}_{t})( \cdot \times \cdot )\\
	\nonumber  &=& \int_{\mathcal{S}(\RR^2)} \frac{1}{2} (\Delta^e \otimes \Delta^e)( \cdot \times \cdot )\, \mathcal{I}^1(s^2 \exp(-s\Lambda_p([e]));t) \ll (\Lambda_p\times\Lambda_p)\cap \Lambda_p\gg (\dint e)\\
	\nonumber && + \int_{\mathcal{S}(\RR^2)} (\Delta^e \otimes \Lambda_p([ \cdot \cap e]) + \Lambda_p([ \cdot \cap e]) \otimes \Delta^e) ( \cdot \times \cdot )\,\mathcal{I}^2(s^2 \exp(-s \Lambda_p([e]));t) \ll (\Lambda_p\times\Lambda_p)\cap \Lambda_p\gg (\dint e)\\
	&& + 4\int_{\mathcal{S}(\RR^2)} (\Lambda_p([ \cdot \cap e]) \otimes \Lambda_p([ \cdot \cap e]))( \cdot \times \cdot )\, \mathcal{I}^3(s^2 \exp(-s \Lambda([e]));t) \ll (\Lambda_p\times\Lambda_p)\cap \Lambda_p\gg (\dint e).
\end{eqnarray}%
Starting with the first of these three summands, we can use the definition of the segment intersection measure as given in \eqref{eq:DefSegIntersecMeasure} with \eqref{eq:LambdaSegmentIntersection} and \eqref{eq:Delta} to see that
\begin{eqnarray} \label{eq:DeltaDelta}
	\nonumber &&\int_{\mathcal{S}(\RR^2)} \frac{1}{2} (\Delta^e \otimes \Delta^e)( \cdot \times \cdot ) \,\mathcal{I}^1(s^2 \exp(-s\Lambda_p([e]));t) \ll (\Lambda_p\times\Lambda_p)\cap \Lambda_p\gg (\dint e)\\
	\nonumber&=&  \displaystyle \int_{[\RR^2]} \int_{[L]} \int_{[L]} \frac{1}{2} (\Delta^{L(L_1,L_2)} \otimes \Delta^{L(L_1,L_2)})( \cdot \times \cdot ) \, \mathcal{I}^1(s^2 \exp(-s\Lambda_p([L(L_1,L_2)]));t) \\
	\nonumber && \qquad \qquad \qquad \qquad \qquad \qquad \qquad  \qquad  \qquad \qquad \quad   \times\, \Lambda_p(\dint L_1) \, \Lambda_p(\dint L_2) \, \Lambda_p(\dint L) \\
	\nonumber &=&  \displaystyle p(1-p)^2 \, \int_{\RR} \int_{\RR} \int_{\RR} \frac{1}{2} \Big((\delta_{(\tau ,\sigma)} +  \delta_{(\vartheta, \sigma)}) \otimes (\delta_{(\tau ,\sigma)} +  \delta_{(\vartheta, \sigma)})\Big)( \cdot \times \cdot )
	\,\mathcal{I}^1(s^2 \exp(-s(1-p)\vert \tau - \vartheta\vert);t)\\
	\nonumber && \qquad \qquad \qquad \qquad \qquad \qquad \qquad \qquad \qquad \qquad \qquad \qquad \qquad  \qquad \qquad \qquad \qquad  \quad  \times \,\dint \vartheta \, \dint \tau\, \dint \sigma\\
	\nonumber \\
	&& +\nonumber  \, \displaystyle (1-p)p^2  \hspace{-0.25ex}\int_{\RR} \int_{\RR} \int_{\RR} \frac{1}{2} \Big((\delta_{(\sigma, \tau)} +  \delta_{(\sigma, \vartheta)}) \otimes (\delta_{(\sigma, \tau)} +  \delta_{(\sigma, \vartheta)})\Big) ( \cdot \times \cdot )\, \mathcal{I}^1(s^2 \exp(-sp\vert \tau - \vartheta\vert);t) \\
	 && \qquad \qquad \qquad \qquad \qquad \qquad \qquad \qquad \qquad \qquad \qquad \qquad \qquad \qquad \qquad \qquad  \quad    \times \, \dint \vartheta \, \dint  \tau \,\dint \sigma.
\end{eqnarray}
We now aim at giving the corresponding Mondrian analogue of the pair-correlation function of the vertex point process. As in the previous sections, we do so by giving the reduced covariance measure via a diagonal-shift argument in the sense of \cite[Corollary 8.1.III]{DaleyVerejones}.
Consider the first integral term in \eqref{eq:DeltaDelta} without its coefficient for the Borel set $ A\times B \subset\RR^2\times\RR^2$. After multiplying the Dirac measures, we only consider the first two summands that integrate over $\delta_{(\tau ,\sigma)}(A)\delta_{(\tau ,\sigma)}(B)$ and $\delta_{(\tau ,\sigma)}(A)\delta_{(\vartheta, \sigma)}(B)$, respectively, as the other two can be handled in the same fashion. Using Lemma~\ref{lem:IntegralExpressions}(iii) yields
\begin{eqnarray*}
	&&\int_{\RR} \int_{\RR} \int_{\RR} \, \delta_{(\tau ,\sigma)}(A)\delta_{(\tau ,\sigma)}(B) \,  \mathcal{I}^1\big(s^2 \exp(-s(1-p)| \tau-\vartheta| ) ;t\big) \,  \dint \vartheta \, \dint \tau \, \dint \sigma\\
	\\
	&=& \int_{\RR^2} \delta_{w}(A)\, \int_{\RR} \delta_{\mathbf{0}}(B - w)  \mathcal{I}^1\big(s^2 \exp(-s(1-p)|-z_1| ) ;t\big) \, \dint z \, \dint w
\end{eqnarray*}
and 
\begin{eqnarray*}
	&&\int_{\RR} \int_{\RR} \int_{\RR} \delta_{(\tau ,\sigma)}(A)\delta_{(\vartheta, \sigma)}(B) \, \mathcal{I}^1\big(s^2 \exp(-s(1-p)| \tau-\vartheta| ) ;t\big) \, \, \dint \vartheta \, \dint \tau \, \dint \sigma \\
	\\
	&=& \int_{\RR^2} \delta_{w}(A) \int_{\RR}  \delta_{(z, 0)}(B - w) \, \mathcal{I}^1\big(s^2 \exp(-s(1-p)|-z|) ;t\big) \, \, \dint z \, \dint w.
\end{eqnarray*}
We get the same terms from the third and fourth summand within the first integral expression, yielding
\begin{eqnarray*}
	&& 2 \int_{\RR^2} \delta_{w}(A)\, \int_{\RR}\Big( \delta_{\mathbf{0}}(B - w) +\delta_{(z, 0)}(B - w) \Big)  \mathcal{I}^1\big(s^2 \exp(-s(1-p)|z| ) ;t\big) \,\dint z\, \dint w
	\end{eqnarray*}
as the overall value of that term. The same argument, with an appropriate swap of $x-$ and $y$-coordinate, shows that the second integral in \eqref{eq:DeltaDelta} (again without its coefficient) is equal to
\begin{eqnarray*}
	&& 2 \int_{\RR^2} \delta_{w}(A)\, \int_{\RR}\Big( \delta_{\mathbf{0}}(B - w)+ \delta_{(0, z)}(B - w)\Big)  \mathcal{I}^1\big(s^2 \exp(-sp|z| ) ;t\big)   \,\dint z\,  \dint w.
\end{eqnarray*}
For the second and third term in \eqref{eq:VertexPairCovMeasure} we see that, due to \eqref{eq:DeltaE} and  \eqref{eq:LambdaSegmentIntersection}, these can be handled  with the help of Lemma~\ref{lem:IntegralExpressions}(i) and (ii), respectively, see also the handling of such terms as carried out in  Equations~\eqref{eq:DoubleIntegralLebesgue}, \eqref{eq:DiracLebesgue1} and \eqref{eq:DiracLebesgue2}. 

We again define a function in the spirit of Ripley's K-function via the reduced second moment measure $\widehat{\mathcal{K}}( \mathcal{V}_{t})(R_{r,p})$ of $R_{r,p}$, $r>0$, and the corresponding normalized derivative as
\begin{align}\label{eq:gPair}
K_{\mathcal{V}}(r)=\frac{1}{\lambda_\mathcal{V}^2} \, \widehat{\mathcal{K}}( \mathcal{V}_{t})(R_{r,p})=\frac{1}{(t^2p(1-p))^2} \, \widehat{\mathcal{K}}( \mathcal{V}_{t})(R_{r,p}). 
\end{align} 
%
Combining the considerations above with the diagonal shift argument from \cite[Corollary 8.1., III]{DaleyVerejones} we obtain that the reduced covariance measure $ \widehat{\Cov}( \mathcal{V}_{t})$ with
\begin{align*}
	\Cov( \mathcal{V}_{t})(A\times B)= \int_{A} \int_{B-x} \widehat{\Cov}( \mathcal{V}_{t})(\dint y)\ \ell_2(\dint x),
\end{align*}
is given by
\begin{eqnarray*} 
	\nonumber&& \widehat{\Cov} ( \mathcal{V}_{t})( \, \cdot \, )
	\nonumber\\
	\nonumber && =p(1-p)\Bigg[(1-p) \bigg(\int_{\RR} \delta_{\mathbf{0}}(\cdot)  \mathcal{I}^1\big(s^2 \exp(-s(1-p)|z| ) ;t\big) \, \dint z\\
	\nonumber && \qquad \qquad  \qquad \qquad  + \int_{\RR}  \delta_{(z, 0)}(\cdot) \, \mathcal{I}^1\big(s^2 \exp(-s(1-p)|z|) ;t\big) \,  \dint z  \bigg)\\
	\nonumber && \qquad + p \bigg( \int_{\RR} \delta_{\mathbf{0}}(\cdot)  \mathcal{I}^1\big(s^2 \exp(-sp|z| ) ;t\big)  \, \dint z+ \,  \int_{\RR}  \delta_{(0, z)}(\cdot) \, \mathcal{I}^1\big(s^2 \exp(-sp|z|) ;t\big) \,   \dint z \, \bigg)\\
	\nonumber\\
	\nonumber && \qquad  + 4\Bigg( (1-p)^2 \int_{\RR}  \,\ell_1( \cdot \cap \overline{(0  z)}_0) \, \, \mathcal{I}^2\big(s^2 \exp(-s(1-p)|z|);t\big)\,   \dint z\\
	\nonumber\\
	\nonumber && \qquad  \qquad \quad + p^2 \int_{\RR}  \,\ell_1( \cdot \cap {_0}\overline{(0  z)}) \, \, \mathcal{I}^2 \big(s^2 \exp(-sp|z|);t\big)\,   \dint z\\
	\nonumber && \qquad \qquad \quad  + (1-p)^3 \int_{\RR} \Big(\int_{ \overline{(0z)}_0}  \int_{  \overline{(0z)}_0}    \delta_{y-x}(\cdot)\, \dint x \, \dint y\Big)\,\mathcal{I}^3\big(s^2 \exp(-s(1-p)\vert z\vert) ;t\big)\,  \dint z\\
	\nonumber\\
 && \qquad \qquad \quad + p^3 \int_{\RR} \Big(\int_{ \overline{{_0}(0z)}}  \int_{  {_0}\overline{(0z)}}    \delta_{y-x}(\cdot)\, \dint x \, \dint y\Big)\,\mathcal{I}^3\big(s^2 \exp(-sp\vert z\vert) ;t\big)\,  \dint z\Bigg) \Bigg].
\end{eqnarray*}
The relationship \cite[Equation (8.1.6)]{DaleyVerejones} now yields the reduced second moment measure $\widehat{\mathcal{K}}( \mathcal{V}_{t})$ 
\begin{eqnarray*} 
	\widehat{\mathcal{K}}( \mathcal{V}_{t})( \, \cdot \, )&=&  \widehat{\Cov}( \mathcal{V}_{t})( \, \cdot \, )+ (t^2p(1-p))^2 \ell_2( \, \cdot \, ).
\end{eqnarray*}
We now use \eqref{eq:DeltaRp}, \eqref{eq:IntersectionRecLineSegE1}, \eqref{eq:IntersectionRecLineSegE2}, and Lemma \ref{lem:DoubleIntegral0zDirac} in straight-forward calculations, to obtain
\begin{align*}\label{eq:K-FctPair}
	&\widehat{\mathcal{K}}( \mathcal{V}_{t})(R_{r,p})\notag\\ 
	 &=: (t^2p(1-p))^2 r^2 p(1-p) +p(1-p) \Big[(1-p)S_1(r)+ pS_2(r)+ 4\Big(S_3(r)+ (1-p)^3S_4(r) +p^3S_5(r)\Big)\Big].
\end{align*}
\allowdisplaybreaks
With the calculations from \eqref{eq:DerivativeIntegralI}, \eqref{eq:DerivativeIntegralII}, \eqref{DerivativeT3} and \eqref{DerivativeT2}, we get the derivatives
\begin{align*}
\frac{\dint S_1(r)}{\dint r}&=\frac{\dint }{\dint r}\Bigg[2\int_{0}^\infty \mathcal{I}^1\big(s^2 \exp(-s(1-p)z ) ;t\big) \, \dint z + \int_{0}^{(1-p)r}   \, \mathcal{I}^1\big(s^2 \exp(-s(1-p)z) ;t\big) \,  \dint z\Bigg]\\
\\
&= (1-p) \,  \mathcal{I}^1\big(s^2 \exp(-s(1-p)^2r) ;t\big)
\intertext{ and }
\frac{\dint S_2(r)}{\dint r}&=\frac{\dint}{\dint r}\Bigg[ 2\int_0^\infty \mathcal{I}^1\big(s^2 \exp(-spz ) ;t\big) \,\dint z + \,  \int_{0}^{pr}   \, \mathcal{I}^1\big(s^2 \exp(-spz) ;t\big) \, \dint z\Bigg]= p \,  \mathcal{I}^1\big(s^2 \exp(-sp^2r) ;t\big).
\end{align*}
The other three terms can be handled in the same way as the terms $ T_1 $, $ T_2 $ and $ T_3  $ in Section~\ref{sec:CrossCorr}. More precisely, $ S_3 $ can be dealt with in the same way as $ T_1 $, $ S_4 $ as $ T_2 $ and $ S_5 $ as $ T_3 $. Finally,  recalling $K_{\mathcal{V}}(r)$ and 
$g_{\mathcal{V}}(r)$ from \eqref{eq:gPair}, and using the concrete integral calculations from Lemma \ref{lem:ValueIntegralCalc}, we conclude the proof of Theorem \ref{thm:VertexPair}.\qed
\section*{Acknowledgements}
The authors would like to thank Claudia Redenbach for providing the helpful images of planar Mondrian tesselations with different weights in Figure \ref{fig:WeightedMondrian}. CB and CT were supported by the DFG priority program SPP 2265 \textit{Random Geometric Systems}.

\addcontentsline{toc}{section}{References}


\bibliographystyle{abbrv}
\bibliography{RectSTITsBibliography}
\end{document}